\documentclass[a4paper,10pt,reqno]{amsart}
\usepackage{amssymb,latexsym,amsmath,amsfonts}
\usepackage{amsthm}
\usepackage[mathscr]{eucal}
\usepackage{rotating}
\usepackage{multirow}
\usepackage[left=1in, right=1in, bottom=1.25in, top = 1.25in]{geometry}
\newtheorem{theorem}{Theorem}
\usepackage[colorlinks=true, citecolor=blue]{hyperref}
\newtheorem{lemma}{Lemma}
\newtheorem{proposition}{Proposition}
\newtheorem{corollary}{Corollary}

\theoremstyle{definition}
\newtheorem{definition}[theorem]{Definition}

\newtheorem{example}[theorem]{Example}
\theoremstyle{remark}
\newtheorem{remark}{Remark}

\numberwithin{equation}{section}
\numberwithin{theorem}{section}
\numberwithin{proposition}{section}
\begin{document}

\title[Large Fluctuations in Volterra Summation Equations]
{Large fluctuations and growth rates of linear Volterra summation equations}

\author{John A. D. Appleby}
\address{School of Mathematical
Sciences, Dublin City University, Glasnevin, Dublin 9, Ireland}
\email{john.appleby@dcu.ie} \urladdr{webpages.dcu.ie/\textasciitilde
applebyj}

\author{Denis D. Patterson}
\address{School of Mathematical
Sciences, Dublin City University, Glasnevin, Dublin 9, Ireland}
\email{denis.patterson2@mail.dcu.ie}
\urladdr{sites.google.com/a/mail.dcu.ie/denis-patterson}

\thanks{Denis Patterson is supported by the Irish Research Council grant GOIPG/2013/402.} 
\date{14th October 2016}

\begin{abstract}
This paper concerns the asymptotic behaviour of solutions of a linear convolution Volterra summation equation with an unbounded forcing term. In particular, we suppose the kernel is summable and ascribe growth bounds to the exogenous perturbation. If the forcing term grows at a geometric rate asymptotically or is bounded by a geometric sequence, then the solution (appropriately scaled) omits a convenient asymptotic representation. Moreover, this representation is used to show that additional growth properties of the perturbation are   preserved in the solution. If the forcing term fluctuates asymptotically, we prove that fluctuations of the same magnitude will be present in the solution and we also connect the finiteness of time averages of the solution with those of the perturbation. Our results, and corollaries thereof, apply to stochastic as well as deterministic equations, and we demonstrate this by studying some representative classes of examples. Finally, we show that our theory can be extended to cover a class of nonlinear equations via a straightforward linearisation argument.
\end{abstract}

\maketitle

\section{Introduction}
Volterra equations, both discrete and continuous, have found myriad applications in the field of economics and it is this area of application we have in mind throughout the present work. In the context of economic growth models it is especially pertinent to analyze qualitative features of unbounded solutions, such as growth rates and fluctuation sizes, and hence we determine conditions under which the forced Volterra summation equation 
\begin{equation} \label{eq.x}
x(n+1)= {\sum_{j=0}^{n}} k(n-j) x(j) + H(n+1),\quad n\geq 0; \quad x(0)=\xi\in \mathbb{R}
\end{equation}
has unbounded solutions with additional growth properties. 

We assume henceforth that 
\begin{equation} \label{eq.k}
k\in \ell^1(\mathbb{Z}^+) 
\end{equation}
and that the unperturbed, or resolvent, equation 
\begin{equation} \label{eq.r}
r(n+1)=\sum^{n}_{j=0}k(n-j) r(j), \quad n\geq 0; \quad r(0)=1.
\end{equation}
has a summable solution, that is
\begin{equation} \label{eq.rinl1}
r\in \ell^1(\mathbb{Z}^+).
\end{equation}
Of course, the summability of $r$ can be characterised entirely in terms of the kernel $k$; in particular, $r\in \ell^1(\mathbb{Z}^+)$ is equivalent to the characteristic equation condition
\begin{equation} \label{eq.chareqn}
	1-\sum^\infty_{l=0}k(l) z^{-(l+1)}\neq 0, \quad \text{ for all $z \in \mathbb{C}$ with }|z|\geq 1.
\end{equation}
In the important and special case that $k(n)\geq 0$ for all $n\geq 0$, the condition \eqref{eq.chareqn} is equivalent to 
$
\sum_{j=0}^\infty k(j)<1.
$ Hence, a useful and sharp sufficient condition for $r$ to be summable, which does not require sign conditions on $k$, is $\sum_{j=0}^\infty |k(j)|<1$.

The summability of $r$ is intimately related to the boundedness of the solution to \eqref{eq.x} under bounded perturbations. In fact, it is true that: 
\begin{itemize}
\item[(a.)] If $r$ is summable, then $x$ is bounded if and only if $H$ is bounded,
\item[(b.)] If, for every bounded sequence $H$, $x$ is bounded, then $r$ is summable, 
\end{itemize}
(see Corduneanu \cite{corduneanu1973integral} and Perron \cite{perron1930stabilitatsfrage}). From (a.) it is clear that if $H$ is unbounded, then so is $x$. In this paper, we seek to understand how more refined properties of unbounded forcing sequences $H$ give rise to corresponding unboundedness properties of $x$ and in this sense our results are related to classic admissibility theory for Volterra equations. The right--hand side of \eqref{eq.x} defines a linear Volterra operator and if this operator maps a space $S$ onto itself then $S$ is said to be admissible with respect to the operator (cf. \cite{gyori2012boundedness,gyHori2010asymptotically,gyHori2008new,gyHori2009sharp,gyHori2010admissibility,song2004linearized,song2006admissibility}). Typical admissibility results for operators of the type considered in this paper assert that for every $H \in S$, we will have $x \in S$ (cf. Gripenberg et al. \cite[Theorem 2.4.5]{GLS}), where $S$ is a ``standard'' space (such as the space of bounded, convergent, periodic, or $\ell^p$ sequences). In this context, the main contribution of the present work is to expand the collection of admissible spaces for the discrete Volterra operator defined by \eqref{eq.x} to spaces more germane to economic applications (as opposed to the classic theory which studies spaces more appropriate for applications in engineering and related areas). In particular, we show that if the unbounded sequence $H$ has an interesting property $A$ which characterises its growth or fluctuation, then $x$ possesses the property $A$ as well; in many situations the converse also holds (cf. Appleby and Patterson \cite{appleby2016unbounded}).

We investigate both bounds on the fluctuations of solutions, and on exact rates of growth. 
$H$ is assumed to be unbounded, but its growth bounds are characterised, in the sense that there is an increasing 
sequence $(a(n))_{n \geq 0}$ with $a(n)\to\infty$ as $n\to\infty$ such that 
\begin{equation}\label{intro_Lambda_a}
\limsup_{n\to\infty} \frac{|H(n)|}{a(n)} =:\Lambda_a|H|.
\end{equation}
It is already known in the case when $a(n)\to +\infty$ as $n\to\infty$ (and is increasing) that $\Lambda_a|H|$ finite implies $\Lambda_a|x|$ finite (see, for example, Gol'denger{\^s}hel' \cite{gol1966discrete}). We show that the existence of finite, zero, or infinite values of $\Lambda_a|x|$ and $\Lambda_a|H|$ are closely linked. Specifically, let $a$ be a monotone sequence with $a(n)\to\infty$ as $n\to\infty$, and define the sequence spaces 
\begin{align*}
B_a&=\left\{ (g(n))_{n\geq 0}: \limsup_{n\to\infty} |g(n)|/a(n)<+\infty\right\}, \\
B_a(0)&=\left\{ (g(n))_{n\geq 0}: \limsup_{n\to\infty} |g(n)|/a(n)=0\right\}, \\
B_a(+)&=\left\{ (g(n))_{n\geq 0}: \limsup_{n\to\infty} |g(n)|/a(n)\in (0,\infty)\right\}, \\
U_a&=\left\{ (g(n))_{n\geq 0}: \limsup_{n\to\infty } |g(n)|/a(n)=+\infty\right\}.
\end{align*}
We show that $x\in V$ if and only if $H\in V$, where $V$ is one of the spaces listed above (see Theorem \ref{theorem.fluct}). The aforementioned analysis is relatively straightforward and we expand upon it to establish more refined properties of solutions by asking more of the forcing sequence $H$ and the normalising sequence $a$.

Results of the type described above are especially interesting if $H$ is a stochastic process, since they show that the large fluctuations in $H$ determine those in $x$: we cannot get ``smaller'' fluctuations in $x$ than those of $H$, but neither can we get larger ones.
The limits zero and infinity in equation \eqref{intro_Lambda_a} are of special interest in probability theory. We provide examples when $H$ is a sequence of independent and identically distributed random variables in which a sequence $a$ cannot be found so that $\Lambda_a|H|$ is nontrivial (but that bounding sequences $a_\pm$ can be found so that $\Lambda_{a_+}|H|=0$ and $\Lambda_{a_-}|H|=+\infty$).

Moreover, we can characterise an important class of asymptotic growth. To this end, define the space 
\begin{equation}\label{def.Glam}
G_\lambda=\bigg\{(g(n))_{n\geq 0}: \text{$|g(n)|\to \infty$ as $n\to\infty$, }
\lambda:=\lim_{n\to\infty} \frac{g(n-1)}{g(n)}\in[0,1]\bigg\},
\end{equation}
and the following equivalence relation on the space of real--valued sequences:
\begin{definition}\label{equiv_relation}
$(y(n))_{n \geq 0}$ and $(z(n))_{n \geq 0}$ are asymptotically equivalent if $y(n)-z(n) \to 0$ as $n\to\infty$. We write $y(n) \sim z(n)$ as $n\to\infty$, or $y \sim z$, for short.
\end{definition}
We remark that definition \ref{equiv_relation} differs from the standard asymptotic notation in which $y(n) \sim z(n)$ as $n\to\infty$ is taken to mean that $\lim_{n\to\infty}y(n)/z(n)=1$. We do however occasionally employ the usual Landau notation in which ``$y$ is $o(z)$'' means that $y(n)/z(n) \to 0$ as $n\to\infty$.

The space $G_\lambda$ is related to a class of weight functions first introduced by Chover et al. \cite{chover1973functions} and later employed by Appleby et al. \cite{appleby2006exact} to calculate rates of convergence to zero of solutions to linear Volterra convolution problems. Intuitively, we use sequences in $G_\lambda$ to scale unbounded quantities of interest in much the same way time--series or economic growth models are detrended; we are particularly interested in conditions under which economically relevant growth properties are preserved by this scaling process.

It transpires that $x\in G_\lambda$ if and only if $H\in G_\lambda$, and both imply that $\lim_{n\to\infty} x(n)/H(n)$ is finite, nontrivial, and can be computed explicitly in terms of $k$ and $\lambda$ (see Theorem \ref{theorem.growth2}). 

We also consider a larger class of sequences which are bounded by sequences in $G_\lambda$. Define, for $a \in G_\lambda$, the space
\begin{equation}\label{def.BG_lam}
BG_{a,\lambda}=\left\{(g(n))_{n\geq 0}: \frac{g(n)}{a(n)} \sim (\lambda_a g)(n) \text{ is bounded}\right\}.
\end{equation}
The sequence $\lambda_a g$ in \eqref{def.BG_lam} is only defined up to asymptotic equivalence but one could of course choose $(\lambda_a g)(n) = g(n)/a(n)$ for definiteness. With the notation outlined above, $x\in BG_{a,\lambda}$ if and only if $H\in BG_{a,\lambda}$
and the sequence $\lambda_a x$ can be taken as
\begin{equation} \label{eq.lamxlamh}
(\lambda_a x)(n) \sim (\lambda_a H)(n)+\sum_{j=1}^n r(j)\lambda^j (\lambda_a H)(n-j), \quad \text{ as $n\to\infty$}.
\end{equation}
The result stated above (Theorem \ref{theorem.growth3}) is of particular interest if the forcing term grows in a reasonably regular manner, but has proportional fluctuations around a growth path given by an increasing sequence. This allows, for example, for cyclic growth in $H$ around an exponential trend, leading to similar cyclic growth in $x$ (see Proposition \ref{prop.periodic} and Example \ref{eq.periodic}). 

On the other hand, if $H$ is experiencing fluctuations, and the large fluctuations of $H$ can be crudely bounded by an increasing sequence $a$, the asymptotic relation \eqref{eq.lamxlamh} shows how the large fluctuations in $x$ arise as a ``weighted average'' of the fluctuations in $H$. In a sense, if the decay in $k$ is relatively slow, then $r$ tends to experience a slower decay to zero (see Appleby et al. \cite{appleby2006exact}), so the weight attached to big values of $\lambda_a H$ in the past tends to be larger, and so the present impact of fluctuations in $H$ in the past lingers longer in the fluctuations in $x$. This mechanism also explains how fluctuations around a growth trend in $H$ propagate through to those in $x$. 

We also prove nonlinear variants of each of the growth and fluctuation results outlined above. If the linear term $x(j)$ in the convolution in \eqref{eq.x} is replaced by a nonlinear term $f(x(j))$ such that $f(x)/x\to 1$ as $|x|\to\infty$, then the nonlinear equation thus formed inherits all the growth properties of the underlying linearised equation.

Our proofs primarily rely on the following  variation of constants formula for the \emph{convolution} problem \eqref{eq.x} which allows $x$ to be written directly in terms of $r$ and $H$ (see, for example, Elaydi \cite{elaydi2005introduction}):
\begin{equation} \label{eq.xrep}
x(n)=r(n)x(0)+\sum_{j=1}^n r(n-j)H(j), \quad n\geq 1.
\end{equation}
Since our results frequently involve the quantity $x/a$, where $a$ is a sequence which captures the growth of the unbounded forcing term, another natural line of attack (following Appleby et al. \cite{appleby2006exact} and Reynolds \cite{reynolds2012asymptotic}) would be to rewrite \eqref{eq.x} as follows:
\[
\frac{x(n+1)}{a(n+1)} = \sum_{j=0}^n k(n-j) \frac{x(j)}{a(n+1)} + \frac{H(n+1)}{a(n+1)}, \quad n \geq 0.
\]
We can now let $\tilde{x}(n) = x(n)/a(n)$ for each $n \geq 0$ and consider the \emph{nonconvolution} problem given by
\[
\tilde{x}(n+1) = \sum_{j=0}^n \tilde{k}(n,j) \tilde{x}(j) + \tilde{H}(n+1), \quad n \geq 0,
\]
where $\tilde{k}(n,j) = k(n-j)a(j)/a(n+1)$ for each $n \geq 0$ and $j \in \{0,\dots, n\}$. While this alternative approach would likely lead to more general results, it would also offer weaker, or less precise, conclusions. Hence we prefer to exploit the convolution structure of \eqref{eq.x} and use the formula \eqref{eq.xrep}, as opposed to its nonconvolution analogue (see Vecchio \cite{vecchio2004resolvent}); we believe this approach leads to results more likely to be of use in economic applications where it is of interest to freeze the asymptotically autonomous structure of the equation for the purposes of fitting to data and parsimonious modelling. The aforementioned considerations notwithstanding, it would be interesting to compare the conclusions of the present work with those garnered through the nonconvolution perspective.

The next two sections state and discuss the main results. In Section \ref{sec.growth.economics} we prove two general results on growth of solutions to \eqref{eq.x} and detail some refinements of our main results which we feel are particularly germane in an economic context. We then prove results related to fluctuations and time averages of solutions in Section \ref{sec.fluct.time}. Section \ref{sec.examples} provides practical examples of our theory and outlines applications to stochastic equations. Finally, Section \ref{linearisation} describes the extension of our linear theory to a class of nonlinear equations. The proofs of the main results are largely postponed to the final sections of the paper.

\section{Growth Rates and Economic Applications}\label{sec.growth.economics}
Before stating and discussing our main results we first provide a brief motivation for our interest in equations such as \eqref{eq.x} and outline connections to applications (see also  \cite{appleby2016unbounded} where much of the following discourse is elaborated upon).

When $H \equiv 0$, the right--hand side of \eqref{eq.x} defines an asymptotically autonomous convolution operator. This is particularly desirable in the context of applications where we generally wish to keep the structure of models time independent, not least because we prefer models with time--invariant properties amenable to statistical inference. Another feature of \eqref{eq.x} worth remarking upon is our decision to write ``$H(n+1)$'' as opposed to ``$H(n)$'' when denoting the forcing term. Both formulations are common in the literature but we prefer the former because we are expressly interested in applications to random forcing sequences. In particular, if $H$ is a stochastic process and each random variable $H(n)$ is $\mathcal{F}(n)$--measurable (for each $n \geq 0$), then $x(n)$ is also $\mathcal{F}(n)$--measurable. Hence the value of $x(n)$ is not known with certainty by observers of the system until time $n$; this is the natural setup for problems arising in an economic context (and indeed in most other applications).

One well--known class of economic models which has a formulation closely related to \eqref{eq.x} is the classic dynamic linear multidimensional Leontief input--output model (see Leontief \cite{leontief, leontief1966input}). In the Leontief model, $H$ represents final demand, $x$ is output or production, and the contribution of the convolution term is known as intermediate demand. Time lags reflect the fact that there is a delay between production and satisfaction of final demand. Due to its linear structure, and the exogenous nature of the forcing term, \eqref{eq.x} is also reminiscent of classic time series models. For example, with appropriate choice of $H$, \eqref{eq.x} is particularly closely related to ARMA$(p,q)$ and AR$(\infty)$ models (see, for example, Brockwell and Davis \cite{brockwell2013time}). Such models are often used to capture so--called long range dependence or long memory phenomena, which have been shown to arise in a variety of applied contexts (see Baillie \cite{baillie1996long}, and Ding and Grainger \cite{ding1996modeling}). In contrast to classical analysis of time series models, we focus not on studying the autocovariance function of solutions but on pathwise properties of solutions inherited from the exogenous forcing term.

Our first result below characterises the rate of growth of solutions to \eqref{eq.x} in terms of the sequence space $G_\lambda$, defined by \eqref{def.Glam}.
\begin{theorem} \label{theorem.growth2}
If $k$ and the solution $r$ of \eqref{eq.r} are summable, and $x$ is the solution to \eqref{eq.x}, then the following are equivalent:
\begin{itemize}
\item[(a.)] $H\in G_\lambda$;
\item[(b.)] $x\in G_\lambda$.
\end{itemize}
Moreover, both imply that
\begin{equation} \label{eq.xHlim}
\lim_{n\to\infty} \frac{x(n)}{H(n)}=L,
\end{equation}
where 
\begin{equation}   \label{eq.L}
L=
\frac{1}{1-\sum_{l=0}^\infty \lambda^{l+1} k(l)}, \quad \lambda \in [0,1].
\end{equation}
\end{theorem}
When $\lambda=0$ we see that the sum in \eqref{eq.L} collapses to zero, so that $L=1$. The quantity 
$L$ in \eqref{eq.xHlim} is always nontrivial (and finite) because the summability of $r$ implies 
from \eqref{eq.chareqn} that $1-\sum_{j=0}^\infty k(j)\lambda^{j+1} \neq 0$ for $\lambda\in[0,1]$. When $k$ is positive, there is a ``multiplier effect'' from the input sequence $k$ to the output $x$, because $L>1$ when $\lambda>0$. Equally, there is no multiplier effect if $\lambda=0$. In Section \ref{subsec.growth} we show how Theorem \ref{theorem.growth2} can be used to deal with random forcing sequences which have appropriate structure.

Our next result is similar in spirit to Theorem \ref{theorem.growth2} but deals with more general growth of the type characterised by the space $BG_{a,\lambda}$, defined by \eqref{def.BG_lam}.
\begin{theorem} \label{theorem.growth3}
If $k$ and the solution $r$ of \eqref{eq.r} are summable, and $x$ is the solution to \eqref{eq.x}, then the following are equivalent:
\begin{itemize}
	\item[(a.)] $H\in BG_{a,\lambda}$;
	\item[(b.)] $x\in BG_{a,\lambda}$.
\end{itemize}
Moreover, when $H \in BG_{a,\lambda}$,
\begin{equation}\label{x_asym_relation}
 \frac{x(n)}{a(n)} \sim   (\lambda_a H)(n) + \sum_{j=1}^n r(j)\lambda^j (\lambda_a H)(n-j), \mbox{ as }n\to\infty,
\end{equation}
and similarly, when $x \in BG_{a,\lambda}$,
\begin{equation}\label{H_asym_relation}
\frac{H(n)}{a(n)} \sim (\lambda_a x)(n) - \sum_{j=0}^{n-1} k(j)\lambda^{j+1} (\lambda_a x)(n-j-1), \mbox{ as }n\to\infty.
\end{equation}
\end{theorem}
There is one feature in particular of the asymptotic representations in Theorem \ref{theorem.growth3} which we consider noteworthy. In both \eqref{x_asym_relation} and \eqref{H_asym_relation} the summands decay rapidly to zero if $\lambda \in (0,1)$ and hence the first few terms in the resolvent/kernel sequence are most important. The kernel is in principle known (or could be approximated by time series techniques in the case that $H$ is a stationary process), so there is no difficulty with regard to \eqref{H_asym_relation}. However, to the best of our knowledge, there is limited theory for small time or transient behaviour of the resolvent, although some global results regarding the differential resolvent are available (see, for example, Gripenberg et al. \cite[Theorem 5.4.1]{GLS}). Of course, in practice, the first few terms of $(r(n))_{n \geq 0}$ can be easily calculated by hand and this will likely provide sufficient insight in many instances.

As we will see momentarily, the main strengths of Theorem \ref{theorem.growth3} are its generality and the convenient asymptotic representations \eqref{x_asym_relation} and \eqref{H_asym_relation}. We now extend Theorem \ref{theorem.growth3} by showing that $H \in U \subset BG_{a,\lambda}$ if and only if $x \in U \subset BG_{a,\lambda}$, where the set $U$ is endowed with additional growth properties. 

We first address a type of ``periodic--growth'' which can be thought of as modeling the effect of economic cycles on the long term growth rate of an economy. The following definition of an almost periodic sequence is standard in the literature (see, for example, Agarwal \cite{agarwal2000difference}, or, for a more detailed exposition, Corduneanu \cite{corduneanu1989almost}).
\begin{definition}
A sequence $\pi = (\pi(n))_{n \in \mathbb{Z}}$ is almost periodic if for each $\epsilon>0$ there exists an integer $X(\epsilon)$ such that in any set of $X$ consecutive integers there exists an integer $N$ such that 
\[
|\pi(n+N) - \pi(n)| < \epsilon, \mbox{ for each }n \in \mathbb{Z}.
\] 
We write $\pi \in \text{AP}(\mathbb{Z})$ for short.
\end{definition}
The following definition of an asymptotically almost periodic sequence is also standard (see, for example, Henr{\'i}quez \cite{henriquez1991asymptotically} or Song \cite{song2006almost}).
\begin{definition}
A sequence $\pi = (\pi(n))_{n \geq 0}$ is asymptotically almost periodic if there exists  sequences $\psi \in \text{AP}(\mathbb{Z})$ and $(\phi(n))_{n \geq 0}$ obeying $\phi(n)\to 0$ as $n\to\infty$ such that $\pi(n) = \psi(n)+\phi(n)$ for each $n \geq 0$. We write $\pi \in \text{AAP}(\mathbb{Z}^+)$ for short.
\end{definition}
For $a \in G_\lambda$, define the space of sequences
\begin{equation}\label{BG_pi}
PG_{a,\lambda} = \left\{ (g(n))_{n \geq 0} : \left(g(n)/a(n)\right)_{n \geq 0} \in \text{AAP}(\mathbb{Z}^+)\right\}.
\end{equation}
We are now in a position to state the following result which characterises a type of asymptotic growth incorporating almost periodic cycles.
\begin{proposition}\label{prop.periodic}
If $k$ and the solution $r$ of \eqref{eq.r} are summable, and x is the solution to \eqref{eq.x}, then the following are equivalent:
\begin{itemize}
	\item[(a.)] $H\in PG_{a,\lambda}$;
	\item[(b.)] $x\in PG_{a,\lambda}$.
\end{itemize}
Moreover, when $H \in PG_{a,\lambda}$, the almost periodic part of $x/a$ is given by 
\begin{equation}\label{x_asym_periodic}
\pi_x(n) = \pi_H(n) + \sum_{j=1}^\infty r(j)\lambda^j \pi_H(n-j), \quad n \in \mathbb{Z},
\end{equation}
where $H/a \sim \pi_H \in AP(\mathbb{Z})$. Similarly, when $x \in PG_{a,\lambda}$, the almost periodic part of $H/a$ is given by 
\begin{equation}\label{H_asym_periodic}
\pi_H(n) = \pi_x(n) - \sum_{j=0}^{\infty} k(j)\lambda^{j+1} \pi_x(n-j-1), \quad n \in \mathbb{Z},
\end{equation}
where $x/a \sim \pi_x \in AP(\mathbb{Z})$.
\end{proposition}
The result above is similar to the work of Dibl\'{i}k et al. \cite{diblik2011weighted} in which the authors prove sufficient conditions for the solution of a linear nonconvolution Volterra equation to have an asymptotically periodic solution, when scaled by an appropriate  weight sequence. Moreover, they show that the solution omits an asymptotic representation which identifies the periodic component. In the aforementioned work, and usually in the extant literature (see, for example, Gy\H{o}ri and Reynolds \cite{gyHori2010asymptotically} and the references therein), the solution essentially inherits asymptotic periodicity as a perturbation of an underlying non--delay equation. By contrast, in our result the periodicity of the weighted solution sequence is inherited purely  from the periodic behaviour of the exogenous forcing sequence. 

We also note that the result above holds true if the concept of almost periodicity is replaced simply by standard periodicity; in other words, if the almost periodic part of $H/a$ is periodic, then the almost periodic part of $x/a$ is periodic with the same period (and vice versa). 
\begin{example}\label{eq.periodic}
A simple example in which we can apply Proposition \ref{prop.periodic} is when the forcing term $H$ is endowed with exponential growth around a periodic (or almost periodic) trend. Let $H(n) = \pi^*(n)e^{\alpha n}$ for each $n \geq 0$, where $(\pi^*(n))_{n \in\mathbb{Z}} \in \text{AP}(\mathbb{Z})$ and $\alpha>0$. Naturally, we choose $(a(n))_{n \geq 0} = (e^{\alpha n})_{n \geq 0} \in G_{1/e}$, so that $(H(n)/a(n))_{n \geq 0}$ is bounded. By Proposition \ref{prop.periodic}, $(x(n)e^{-\alpha n})_{n \geq 0} \in \text{AAP}(\mathbb{Z}^+)$ and furthermore,
\[
x(n)e^{-\alpha n} \sim \sum_{j=0}^\infty r(j)e^{-j} \pi^*(n-j), \mbox{ as }n\to\infty.
\]
\end{example}
In the same spirit as the previous result, we consider the case when the forcing sequence $H$ in \eqref{eq.x} has a stable time average when appropriately scaled by a sequence in $G_\lambda$. Given a sequence $a \in G_\lambda$ and another sequence $g$, define the weighted average sequence $(\mu_a g(n))_{n \geq 1}$ by
\[
(\mu_a g)(n) = \frac{1}{n}\sum_{j = 1}^n \frac{g(j)}{a(j)}, \quad\text{for each } n \geq 1.
\]
Now, for $a \in G_\lambda$, define the space of sequences
\begin{equation*}
AG_{a,\lambda} = \left\{ (g(n))_{n \geq 0} : \frac{g(n)}{a(n)} \sim (\lambda_a g)(n) \mbox{ is bounded and}\lim_{n\to\infty} (\mu_a g)(n) \mbox{ exists} \right\}.
\end{equation*}
\begin{proposition}\label{prop.ergodic}
	If $k$ and the solution $r$ of \eqref{eq.r} are summable, and x is the solution to \eqref{eq.x}, then the following are equivalent:
	\begin{itemize}
		\item[(a.)] $H\in AG_{a,\lambda}$;
		\item[(b.)] $x\in AG_{a,\lambda}$.
	\end{itemize}
	Moreover, if $H \in AG_{a,\lambda}$ and $\lim_{n\to\infty}(\mu_a H)(n) =: {\mu_a H}^*$, then
	\[
	\lim_{n\to\infty}(\mu_a x)(n) = \frac{{\mu_a H}^*}{1 - \sum_{j = 0}^\infty. k(j)\lambda^{j+1}},
	\]
	Similarly, if $x \in AG_{a,\lambda}$ and $\lim_{n\to\infty}(\mu_a x)(n) =: {\mu_a x}^*$, then
	\[
	\lim_{n\to\infty}(\mu_a H)(n) = {\mu_a x}^*\left(1 - \sum_{j = 0}^\infty k(j)\lambda^{j+1}\right).
	\]
\end{proposition}
The most natural context for the result above is when the forcing sequence $H$ is random and an ergodic theorem can be applied to conclude that $H$ has stable time averages; we consider an elementary example of this type and invite the interested reader to consider others of similar character.
\begin{example}
Suppose $(\Omega,\mathcal{B}(\mathbb{R}),\mathbb{P})$ is a probability space and $(h(n))_{n \geq 0}$ is a stationary sequence of random variables (not necessarily independent). In other words,
\[
\mathbb{P}[h(n_1) \in B_1,\dots,h(n_r) \in B_r] = \mathbb{P}[h(n_1+k) \in B_1,\dots,h(n_r+k) \in B_r]
\]
for any nonnegative integers $n_1 < n_2 < \ldots < n_r$, Borel sets $B_1,\ldots,B_r$, and positive integer $k$. Suppose further that the $h(n)$'s are bounded on the a.s. event $\Omega_1$. Let $H(n) = h(n)\alpha^n$ for each $n \geq 0$ and some $\alpha \in (1,\infty)$. Hence we may choose $(a(n))_{n\geq 0} = (\alpha^n)_{n \geq 0} \in G_{1/\alpha}$, so that $(H(n)/a(n))_{n \geq 0}$ is bounded on $\Omega_1$. By Birkhoff's ergodic theorem, $\lim_{n\to\infty}\tfrac{1}{n}\sum_{j=1}^n h(j) =: {\mu_a H}^*$ exists on an event of probability one, say $\Omega_2$. Applying Proposition \ref{prop.ergodic} on the a.s. event $\Omega := \Omega_1 \cap \Omega_2$ yields
\[
\lim_{n\to \infty}\frac{1}{n}\sum_{j=1}^n \alpha^{-j}x(j) = \frac{{\mu_a H}^*}{1 - \sum_{j = 0}^\infty k(j)\alpha^{-(j+1)}} \quad\mbox{a.s.}
\]
\end{example}
\section{Fluctuations and Time Averages}\label{sec.fluct.time}
We now explore fluctuations of the solutions of \eqref{eq.x}. 
Suppose that $(a(n))_{n \geq 0}$ is an increasing sequence with $a(n)\to\infty$ as $n\to\infty$, and for any sequence $(y(n))_{n\geq 0}$ define 
\begin{equation} \label{def.Lambdaya}
\Lambda_a|y|=\limsup_{n\to\infty} \frac{|y(n)|}{a(n)}. 
\end{equation}
Our first result illustrates the close coupling of the quantities $\Lambda_a |x|$ and $\Lambda_a |H|$ when $k$ and $r$ are summable. The authors previously obtained similar results for a nonlinear analogue of \eqref{eq.x} with summable kernel and sublinear nonlinearity \cite[Section 3.5]{appleby2016unbounded}.
\begin{theorem} \label{theorem.fluct}
Suppose $a$ is an increasing sequence with $a(n)\to\infty$ as $n\to\infty$. If $k$ and the solution $r$ of \eqref{eq.r} are summable, and x is the solution to \eqref{eq.x}, then
\begin{itemize}
\item[(a.)] $\Lambda_a |x|=0$ if and only if $\Lambda_a |H|=0$; 
\item[(b.)] $\Lambda_a |x|\in (0,\infty)$ if and only if $\Lambda_a |H|\in (0,\infty)$; 
\item[(c.)] $\Lambda_a |x|=+\infty$ if and only if $\Lambda_a |H|=+\infty$.
\end{itemize}
\end{theorem}
In case (a.) of the result above we see that any fluctuations in $x$ can be no larger than the size of the fluctuations present in $H$, and vice versa. Similarly, if we observe fluctuations of a certain order of magnitude in either $x$ or $H$, fluctuations of the same order must have been present in the other sequence (case (b.)). Finally, fluctuations in the forcing sequence $H$ must lead to fluctuations at least as large in the solution sequence $x$, and vice versa. While still applicable in a deterministic setting, the result above is more natural and useful when the forcing sequence is random and we employ Theorem \ref{theorem.fluct} in this context in Section \ref{subsec.fluctuation}.

The next set of results provide bounds on time--averaged functionals of the solution to \eqref{eq.x}. The need for results of this type arises frequently in a variety of applications, particularly when $H$ is a stochastic process (see Appleby and Patterson \cite[Section 3.8]{appleby2016unbounded} for sample applications).  
\begin{theorem}\label{thm.phi.moments}
Suppose $k$ and the solution $r$ of \eqref{eq.r} are summable, and $x$ is the solution to \eqref{eq.x}. If $\phi:[0,\infty) \mapsto [0,\infty)$ is an increasing convex function, then
\[
\limsup_{n\to\infty}\frac{1}{n}\sum_{j=0}^n\phi(|x(j)|) \leq \limsup_{n\to\infty}\frac{1}{n}\sum_{j=0}^n\phi\left(|r|_1\,|H(j)|\right).
\]
Similarly, 
\[
\limsup_{n\to\infty}\frac{1}{n}\sum_{j=0}^n\phi(|H(j)|) \leq \limsup_{n\to\infty}\frac{1}{n}\sum_{j=0}^n\phi\left((1+|k|_1)|x(j)|\right).
\]
\end{theorem}
We remark that the non--unit multipliers inside the argument of $\phi$ in the result above are not entirely an artefact of the method of proof, although neither is our estimate sharp in general. We illustrate this point with a short example in which the sequence $H$ is random. As usual, we defer the proof to the end.
\begin{example}\label{eg.non_unit_multiplier}
Let $\sigma>0$ and suppose $H$ is a sequence of independent and identically distributed normal  random variables with mean zero and variance $\sigma^2$. Let $x(0)$ be deterministic and suppose  $r \in \ell^1(\mathbb{Z}^+)$. By the strong law of large numbers,
\[
\lim_{n\to\infty}\frac{1}{n}\sum_{j = 0}^n H^2(j) = \sigma^2 \mbox{ a.s.}
\]
Furthermore, if $\lim_{n\to\infty}\log(n)\sum_{j = n}^\infty r^2(j) = 0$, then it can be shown that
\[
\lim_{n\to\infty}\frac{1}{n}\sum_{j = 0}^n x^2(j) = \sigma^2 \sum_{j = 0}^\infty r^2(j) = \sigma^2 \left(|r|_2\right)^2\mbox{ a.s.},
\]
where $|r|_2$ denotes the $\ell^2$--norm of $r$. In the context of this example, applying Theorem \ref{thm.phi.moments} with $\phi(x)=x^2$ would enable us to conclude that
\[
\limsup_{n\to\infty}\frac{1}{n}\sum_{j = 0}^n x^2(j) \leq \left(|r|_1\right)^2\limsup_{n\to\infty}\frac{1}{n}\sum_{j = 0}^n H^2(j) = \sigma^2 \left(|r|_1\right)^2\mbox{ a.s.}
\]
Since $\left(|r|_2\right)^2 \leq \left(|r|_1\right)^2$ in general, usually with strict inequality, the aforementioned lack of sharpness in the conclusion of Theorem \ref{thm.phi.moments} is immediately apparent.
\end{example}

With a slight strengthening of hypotheses on $\phi$ we can immediately prove a very useful corollary regarding the finiteness of time averaged functionals of the solution to \eqref{eq.x}. We first require the following standard definition.
\begin{definition}
A nonnegative measurable function $\phi$ is called $O$--regularly varying if
\[
0 < \liminf_{x\to\infty}\frac{\phi(\lambda x)}{\phi(x)} \leq \limsup_{x\to\infty}\frac{\phi(\lambda x)}{\phi(x)} < \infty, \mbox{ for each }\lambda>1.
\]
\end{definition} 
While the definition above may seem somewhat restrictive, it turns out that if $\phi$ is increasing and $\limsup_{x\to\infty}\phi(\lambda x)/\phi(x)$ is finite for some $\lambda >1$, then $\phi$ is $O$--regularly varying (see Bingham et al. \cite[Corollary 2.0.6, p.65]{bingham1989regular}). We now state the aforementioned corollary to Theorem \ref{thm.phi.moments} without proof.
\begin{corollary}
Suppose $k$ and $r$ are summable, and x is the solution to \eqref{eq.x}. If $\phi:[0,\infty) \mapsto [0,\infty)$ is an increasing, convex, and $O$--regularly varying function, then the following are equivalent:
\begin{enumerate}
\item[(a.)] \[\limsup_{n\to\infty}\frac{1}{n}\sum_{j=0}^n\phi(|x(j)|) < \infty,\]
\item[(b.)] \[\limsup_{n\to\infty}\frac{1}{n}\sum_{j=0}^n\phi(|H(j)|) < \infty.\]
\end{enumerate}
\end{corollary}
\section{Examples \& Applications to Stochastic Volterra Equations}\label{sec.examples}
\subsection{Growth}\label{subsec.growth}
The space $G_\lambda$ contains many well--behaved sequences covering a wide range of growth, from very slow to very rapid. To see this, it is useful to introduce notation for iterated logarithms and exponentials. For any positive integer $k$, we define inductively the iterated logarithm $\log_k(x)=\log(\log_{k-1}(x))$ for $k\geq 2$ and 
$\log_1(x)=\log(x)$ (for appropriate positive $x$), and the iterated exponential $\exp_k(x)=\exp(\exp_{k-1}(x))$ for $k\geq 2$ and 
$\exp_1(x)=\exp(x)$ for all $x>0$. 

For examples in $G_1$ consider a sequence asymptotic to $H_1(n)$ where 
\[
H_1(n)=\prod_{i=1}^j (\log_{i}(n))^{\beta_i}
\]
where $j$ is a positive integer, and $(\beta_i)_{i=1}^j$ are any real numbers such that the first non--zero entry in the sequence $\beta$ is positive. We also can take a sequence asymptotic to $H_1(n)$, where $\theta>0$
\[
H_2(n)=n^{\theta_1} H_1(n) 
\]
and there is no restriction on the sequence $\beta$ in $H_1$. $H_3(n)=n^{\theta_1}$ for $\theta_1>0$ is another example in $G_1$. Sequences which grow faster than positive powers of $n$, but slower than exponentially are also admissible. For instance, sequences asymptotic to 
\[
H_4(n)=\exp(\alpha n^{\theta_2})
\] 
for $\alpha>0$ and $\theta_2\in (0,1)$, or even asymptotic to $H_5(n)=H_4(n)H_2(n)$ (without restriction on $\theta_1$ or $\beta$ in $H_2$), are also in  $G_1$. For examples of sequences in $G_\lambda$ for $\lambda\in (0,1)$, we have sequences which grow geometrically or are dominated by 
a geometric growth component. Such sequences include those asymptotic to $H_6(n)=\lambda^{-n}$ or 
\[
H_7(n)=H_6(n) H_4(n)H_2(n)
\]
where there is no restriction on $\theta_1$ or $\beta$ in $H_2$, nor on $\theta_1$ in $H_4$. Finally it can be seen that $G_0$ contains many sequences which grow faster than any geometric sequence. Examples include
\[
H_8(n)=H_4(n)
\]
where $\alpha>0$ and $\theta_2>1$, $H_{9}(n) = n!$, and $H_{10}(n)=\exp_j(n)$ for any integer $j\geq 2$. 
\begin{example} \label{examp.SRWdrift}
The sequences considered so far are deterministic and growing, and within the class of growing stochastic processes, some sequences reside in $G_\lambda$ while others do not. Consider for instance the random walk with drift 
\[
H(n)=\mu n + \sum_{j=1}^n Y(j), \quad n\geq 1
\] 
where $\mu\neq 0$ and the $Y$'s are independent and identically distributed random variables with $\mathbb{E}[Y(j)]=0$ and 
$\mathbb{E}[|Y(j)|]=:\mu_1<+\infty$. We further assume that  
	\begin{equation} \label{eq.DelHnondegen}
	\mathbb{P}[Y(n)=c]<1 \text{ for all $c\in \mathbb{R}$}, 
	\end{equation}
	so that the $Y$'s are meaningfully random, and are not almost surely constant.
	By the strong law of large numbers, 
	\[
	\lim_{n\to\infty} \frac{H(n)}{n}=\mu\neq 0, \quad \text{a.s.}
	\] 
so $|H(n)|$ is asymptotic to the sequence $|\mu|n$ which is increasing, and $H$ is clearly in $G_1$, despite the fact that the sequence $|H|$ is not monotone increasing or even ultimately monotone. To see this, take without loss $\mu>0$ and suppose that 
	$\mathbb{P}[Y(j)<-\mu]>0$. This will certainly be true if each $Y$ has a distribution supported on all $\mathbb{R}$. Therefore 
	\[
	\mathbb{P}[H(n+1)-H(n)<0]=\mathbb{P}[\mu+Y(n+1)<0]>0,
	\]   
so at each step there is a constant and positive probability that $H$ decreases. Moreover, as the events  $\{Y(j)<-\mu\}$ are independent, it is true that 
	\[
	\mathbb{P}[H(n+1)<H(n) \text{ for some $n>m$}]=1, \mbox{ for each $m\in \mathbb{N}$},
	\]
so not only is $H$ non--monotone almost surely, it is also ultimately non--monotone almost surely.
\end{example}
	
\begin{example} \label{examp.GRW}
An example of a sequence that is not in $G_\lambda$ (in general) is the geometric random walk. Suppose $\mu>0$ above and consider 
the sequence
\begin{equation} \label{eq.GRW}
H(n)= \exp\left(\mu n + \sum_{j=1}^n Y(j)\right), \quad n\geq 1
\end{equation}
where the process $Y$ is as in Example \ref{examp.SRWdrift}. Clearly,
\[
\frac{H(n-1)}{H(n)}=  e^{-\mu} e^{- Y(n)},
\]
so $\lim_{n\to\infty} H(n-1)/H(n)$ exists if and only if $Y(n)$ tends to a finite limit. But as the sequence $Y$ consists of independent and identically distributed random variables, which obey the non--degeneracy condition \eqref{eq.DelHnondegen}, this is impossible. Hence, the growing geometric random walk is not in $G_\lambda$, for any $\lambda\in [0,1]$. 
\end{example}
To determine the asymptotic behaviour of $x$ for less regular forcing sequences requires further assumptions on the data, and weaker conclusions. Here is an example of the type of result that can be established: we work with the geometric random walk from \eqref{eq.GRW} for definiteness.
\begin{theorem} \label{thm.GRW}
Suppose that $\mu>0$ and that $H$ is the geometric random walk given by \eqref{eq.GRW}. Let $x$ denote the solution of \eqref{eq.x} and suppose $x(0)>0$. If $k$ is non--negative, with $\sum_{j=0}^\infty k(j)<1$, then $x(n)>0$ for all $n\geq 0$ a.s., $x(n)\to\infty$ as $n\to\infty$ a.s. and 
\[
\lim_{n\to\infty} \frac{1}{n} \log x(n)=\mu, \quad\text{a.s.}
\] 
\end{theorem}
\begin{proof}
Consider the sequence $(\log H(n))_{n \geq 0}$ and apply the strong law of large numbers to obtain
\[
\lim_{n\to\infty} \frac{1}{n}\log H(n)=\mu, \quad \text{a.s.}
\]
Suppose that $\Omega^\ast$ is the almost sure event on which this limit prevails. Since $k$ is non--negative, we have that $r$ is summable because  $\sum_{j=0}^\infty k(j)<1$. Clearly we have that $x(n)> 0$ for all $n\geq 0$ a.s. and furthermore that $x(n)>H(n)$ for all $n\geq 1$ a.s. Hence 
\[
\liminf_{n\to\infty} \frac{1}{n}\log x(n)\geq \mu, \quad\text{a.s.}
\]
On the other hand, for every $\omega\in \Omega^\ast$ and $\epsilon>0$, $H(n)<e^{(\mu+\epsilon)n}=:h_\epsilon(n)$ for all $n\geq N(\epsilon,\omega)$. Then we have
\[
x(n+1)<h_\epsilon(n+1)+\sum_{j=0}^n k(n-j)x(j), \quad n\geq N(\epsilon,\omega)+1.
\]
Define $x^\ast(\epsilon)=\max_{0\leq j\leq N(\epsilon,\omega)+1} x(j)$, so that
\[
x(n+1)<x^\ast(\epsilon), \quad n=0,\ldots,N(\epsilon,\omega).
\] 
Hence, with $H_\epsilon(n):=e^{(\mu+\epsilon)n}+x^\ast(\epsilon)$ for all $n\geq 0$, we have the inequality
\[
x(n+1)<H_\epsilon(n+1)+\sum_{j=0}^n k(n-j)x(j), \quad n\geq 0; \quad x(0)=\xi>0.
\]
Now, consider the solution of the summation equation
\[
x_\epsilon(n+1)=H_\epsilon(n+1)+\sum_{j=0}^n k(n-j)x_\epsilon(j), \quad n\geq 0; \quad x_\epsilon(0)=x(0)+1.
\]
By construction, $x(n)<x_\epsilon(n)$ for all $n\geq 0$. Moreover, $x_\epsilon$ omits the representation 
\[
x_\epsilon(n)=r(n)[\xi+1] + \sum_{j=1}^n r(n-j)H_\epsilon(j), \quad n\geq 0,
\]
from \eqref{eq.xrep}. Notice now that $H_\epsilon$ is in $G_\lambda$ for $\lambda=e^{-(\mu+\epsilon)}$. By Theorem~\ref{theorem.growth2}, 
\[
\lim_{n\to\infty} \frac{x_\epsilon(n)}{e^{(\mu+\epsilon)n}} = \frac{1}{1-\sum_{j=0}^\infty k(j)e^{-(\mu+\epsilon)(j+1)}}=:L(\epsilon).
\] 
Hence, for each $\omega\in \Omega^\ast$ and $\epsilon>0$, we have 
\[
\limsup_{n\to\infty} \frac{x(n,\omega)}{e^{(\mu+\epsilon)n}}\leq L(\epsilon).
\]
Therefore
\[
\limsup_{n\to\infty} \frac{1}{n}\log x(n,\omega)\leq \mu+\epsilon, \quad \mbox{for each }\omega\in \Omega^\ast. 
\]
Finally, letting $\epsilon\to 0^+$ in the equation above and combining with the limit inferior yields
\[
\lim_{n\to\infty} \frac{1}{n}\log x(n) =\mu, \quad \text{a.s.}
\]
\end{proof}
\subsection{Fluctuation}\label{subsec.fluctuation}
We first sketch a general framework for dealing with forcing sequences comprised of independent and identically distributed (i.i.d.) random variables and then demonstrate how Theorem \ref{theorem.fluct} can be applied in the presence of such stochastic perturbations. 

Suppose $H$ is a sequence of i.i.d. random variables with common distribution function $F$. For ease of exposition assume the distribution is continuous and supported on all of $\mathbb{R}$.

Since each random variable $H(n)$ has distribution function $F$ we have  
\[
\mathbb{P}[|H(n)|>Ka(n)]=1-F(Ka(n))+F(-Ka(n)).
\]
For each $K \in (0,\infty)$ and sequence $(a(n))_{n \geq 0}$, define 
\begin{equation}\label{S_defn}
S(a,K)=\sum_{n=0}^\infty \{1-F(Ka(n))+F(-Ka(n))  \}.
\end{equation}
Since the events $\{|H(n)|>Ka(n)\}$ are independent, the Borel--Cantelli Lemma implies that 
\[
\mathbb{P}[|H(n)|>Ka(n) \text{ i.o.}]
=\left\{\begin{array}{cc}
0, & \text{if $S(a,K)<+\infty$}, \\
1, & \text{if $S(a,K)=+\infty$}. 
\end{array}
\right.
\]
Therefore, for each $K>0$ such that $S(a,K)<+\infty$, there is an a.s. event $\Omega_K^+$ such that
\[
\Lambda_a|H|=\limsup_{n\to\infty} \frac{|H(n)|}{a(n)} \leq K, \quad\text{on $\Omega_K^+$}.
\]
Similarly, for each $K>0$ such that  $S(a,K)=+\infty$ we have that there is an a.s. event $\Omega_K^-$ such that
\[
\Lambda_a|H|=
\limsup_{n\to\infty} \frac{|H(n)|}{a(n)} \geq K, \quad\text{on $\Omega_K^-$}.
\]
It is sometimes possible, for a carefully--chosen sequence $a$ and number $K$, to produce a sequence $Ka(n)$ for which $S(a,K)$ is either finite or infinite. This will generate upper and lower bounds on the growth of the sequence $(|H(n)|)_{n \geq 0}$, and thereby (via Theorem~\ref{theorem.fluct}) allow conclusions about the growth of the fluctuations of $x$ to be deduced.

We first present a ubiquitous example in which one can find a sequence $a$ for which $\Lambda_a|H|\in (0,\infty)$; we subsequently outline a result which generalises the conclusion of this example to a class of distributions with appropriately fast decay in the tails. 
\begin{example} \label{examp.anormal}
Suppose that $H(n)$ is a sequence of independent normal random variables with mean zero and variance $\sigma^2>0$. If $a(n)=\sqrt{2 \log n}$, then it is well--known that, for each $\epsilon\in (0,\sigma)$, we have  
\[
S(a,\sigma+\epsilon)<+\infty, \quad S(a,\sigma-\epsilon)=+\infty.
\]
Therefore, there are a.s. events $\Omega_\epsilon^{\pm}$ such that 
\[
\limsup_{n\to\infty} \frac{|H(n)|}{\sqrt{2\log n}} \geq \sigma-\epsilon, \quad\text{a.s. on $\Omega_\epsilon^-$\,\,\, and},\quad \limsup_{n\to\infty} \frac{|H(n)|}{\sqrt{2\log n}} \leq \sigma+\epsilon, \quad \text{a.s. on $\Omega_\epsilon^+$}. 
\]
Now consider $
\Omega^\ast=\{\cap_{\epsilon\in \mathbb{Q}\cap(0,\sigma)} \Omega_\epsilon^+\} \cap \{\cap_{\epsilon\in \mathbb{Q}\cap(0,\sigma)} \Omega_\epsilon^-\}$. By construction, $\Omega^\ast$ is an almost sure event and
\begin{align}\label{BC_exact}
\limsup_{n\to\infty} \frac{|H(n)|}{\sqrt{2\log n}} = \sigma, \quad \text{on $\Omega^\ast$}. 
\end{align}
Hence we can apply Theorem~\ref{theorem.fluct} to \eqref{eq.x} with $a(n)=\sqrt{2\log n}$ to obtain
\[
0<\limsup_{n\to\infty} \frac{|x(n)|}{\sqrt{2\log n}} <+\infty, \quad \text{a.s.}
\]
In fact, scrutiny of the proof of Theorem~\ref{theorem.fluct} shows that there are deterministic constants $K_1$ and $K_2$ which depend on 
$K$ (but not on $\sigma$) such that
\[
K_1\sigma\leq \limsup_{n\to\infty} \frac{|x(n)|}{\sqrt{2\log n}} \leq K_2\sigma, \quad \text{a.s.}
\] 
\end{example}
In fact, we can prove a more general result in the case when $(H(n))_{n \geq 0}$ is a sequence of i.i.d. random variables with appropriately ``thin tails''. We first define the class of super--slowly varying functions as follows:
\begin{definition}
A measurable function $\ell$ is called $\xi$-super--slowly varying (at infinity) if the limit
\[
\lim_{x\to\infty}\frac{\ell(x\xi^\delta(x))}{\ell(x)} = 1 
\]
holds uniformly for some $\Delta > 0$ and each $\delta \in [0,\Delta]$. We sometimes write $\ell \in \xi$--SSV for short.
\end{definition}
The class of super--slowly varying functions arises naturally in extreme value theory and, in particular, in the context of inverting asymptotic relations involving slowly varying functions (see Anderson \cite{anderson1978super}, and Bojanic and Seneta \cite{bojanic1971slowly}). The cited works, as well as the classic volume of Bingham et al. \cite[Ch. 3]{bingham1989regular}, give various convenient sufficient conditions for the definition above to hold; in many situations these conditions are preferable to verifying the definition directly. For example, if $\ell$ is slowly varying and $\xi$ is non--decreasing with
\begin{equation}\label{SSV_sufficient}
\frac{\ell(\lambda x)}{\ell(x)} = 1 + o(1/\log\xi(x)), \mbox{ as }x \to \infty, \mbox{ for some }\lambda>1,
\end{equation}
then $\ell$ is $\xi$-super--slowly varying. Similarly, if $\ell$ is continuously differentiable, the condition \eqref{SSV_sufficient} can be replaced by
\[
\frac{x \ell'(x)}{\ell(x)} \mbox{ is } o(1/\log\xi(x)) \mbox{ as }x\to\infty.
\]
We employ super--slow variation here as a convenient way to capture rapid decay in the tails of the distribution function which is sufficiently general that it includes most commonly used thin tailed distributions.
\begin{theorem}\label{rapid_tails_theorem}
Let $(H(n))_{n \geq 0}$ be a sequence of i.i.d. random variables supported on $\mathbb{R}$ with  continuous symmetric distribution function $F$. Define $G(x)=1-F(x)$ for each $x \in \mathbb{R}$. If $G^{-1}(1/x)$ is $\mu$-super--slowly varying as $x\to\infty$ with $\mu:(0,\infty)\mapsto (1,\infty)$ non--decreasing and obeying 
\begin{equation}\label{mu_sum_finite}
\sum_{n=1}^\infty \frac{1}{n \mu^{\delta^*}(n)} < \infty, \mbox{ for some }\delta^* \in (0,\Delta],
\end{equation}
then 
\[
 \limsup_{n\to\infty}\frac{|H(n)|}{G^{-1}\left(1/n\right)} = 1 \quad \mbox{a.s.}
\]
\end{theorem}
\begin{remark}
Under the hypotheses of Theorem \ref{rapid_tails_theorem}, $G^{-1}(1/x)$ is slowly varying as $x\to\infty$, since $\mu(x) \to \infty$ as $x\to\infty$ is a consequence of assuming that $\mu$ is non--decreasing and obeys \eqref{mu_sum_finite}. Hence $G \in \text{RV}_\infty(-\infty)$ and $F \in \text{RV}_{-\infty}(-\infty)$. We also note that symmetry in the tails of the distribution function is not an essential feature of the result stated above and this could be relaxed in the spirit of Theorem \ref{rv_tails_theorem} below.
\end{remark}
It follows that if $H$ is a sequence of random variables satisfying the hypotheses of Theorem \ref{rapid_tails_theorem}, then applying Theorem \ref{theorem.fluct} (which requires that $k$ and $r$ are summable) will show that the solution to \eqref{eq.x} obeys
\[
0 < \limsup_{n\to\infty}\frac{|x(n)|}{G^{-1}(1/n)} < \infty\quad \mbox{a.s.},
\]
generalising the conclusion of Example \ref{examp.anormal}.
\begin{example}
We now provide some straightforward examples of common distributions and slowly varying functions which satisfy the hypotheses of Theorem \ref{rapid_tails_theorem}. For the moment let  $\mu:[0,\infty)\mapsto (1,\infty)$ be an arbitrary increasing and divergent function.

First, we take an example of a function which grows particularly rapidly within the class of slowly varying functions but is still easily seen to be super--slowly varying (one could of course construct a distribution function with corresponding tails). Let $G^{-1}(1/x) /\exp( \log^\alpha x)\to 1$ as $x \to\infty$, with $\alpha \in (0,1)$. Now check the definition of super--slow variation (with $\delta = 1$) directly as follows:
\begin{align*}
\lim_{x\to\infty}\frac{G^{-1}(1/x \mu(x))}{G^{-1}(1/x)} = \lim_{x\to\infty}\exp\left( \log^\alpha (x \mu(x)) - \log^\alpha (x) \right) = \exp\left( \lim_{x\to\infty} \log^\alpha (x \mu(x)) - \log^\alpha (x) \right).
\end{align*}
Define $\xi(x) = \log^\alpha (x \mu(x)) - \log^\alpha (x) = (\log x + \log\mu(x))^\alpha - \log^\alpha (x)$ for $x>1$. If we can choose $\mu$ such that $\lim_{x\to\infty}\xi(x) = 0$, then $G^{-1}(1/x) \in \mu$--SSV. Apply the mean value theorem to $\xi$ to yield
\[
\xi(x) = \alpha (\log x + \theta_x \log\mu(x))^{\alpha-1} \log\mu(x),
\]
for some $\theta_x \in [0,1]$. From this simple estimate we can see that 
\[
\lim_{x\to\infty}\frac{\xi(x)\log^{1-\alpha}x}{\alpha\log\mu(x)} =1.
\]
Hence we could choose $\mu$ such that $\mu(x) / \log^\beta x \to 1$ as $x\to\infty$ with $\beta>1$ and obtain $\lim_{x\to\infty}\xi(x)=0$, as required. This choice would also provide a $(\mu,\delta^*)$ pair satisfying the summability condition \eqref{mu_sum_finite}, since 
\[
\sum_{n=1}^\infty \frac{1}{n \log^\beta(n+1)} < \infty, \mbox{ for each }\beta >1,
\]
by the Cauchy condensation test.

In the case of the normal distribution with variance $\sigma^2$, we have $G^{-1}(1/x) / \sqrt{2\sigma^2\log x} \to 1$ as $x\to\infty$. We once more choose $\delta=1$ and in this instance we require
\begin{align*}
\lim_{x\to\infty}\frac{G^{-1}(1/x \mu(x))}{G^{-1}(1/x)} = \lim_{x\to\infty}\frac{\sqrt{2\sigma^2 \log(x\mu(x))}}{\sqrt{2\sigma^2 \log(x)}} = \sqrt{1 + \lim_{x\to\infty}\frac{\log(\mu(x))}{\log(x)}} = 1.
\end{align*}
A sufficient condition for the limit above to hold is clearly $\lim_{x\to\infty}\log(\mu(x))/\log(x) = 0$. Hence we could choose $\mu$ such that $\mu(x) / \log^\beta x\to 1$ as $x\to\infty$ with $\beta>1$, as in the previous example, and satisfy both condition \eqref{mu_sum_finite} and the requirement that $G^{-1}(1/x) \in \mu$--SSV.

Finally, consider the Weibull distribution which has $G^{-1}(1/x) = \lambda(\log(x))^k$ for $x>1$, with $k$ and $\lambda$ strictly positive constants. This contains the exponential ($k=1$) and Rayleigh ($k=2$) distributions as special cases. Since the Weibull distribution is defined only on the positive half line it does not satisfy the tail symmetry required by Theorem \ref{rapid_tails_theorem} but, as remarked earlier, this can be easily remedied without much additional effort (see Theorem \ref{rv_tails_theorem}). Analogous to the case of the normal distribution, we require
\begin{align*}
\lim_{x\to\infty}\frac{G^{-1}(1/x \mu(x))}{G^{-1}(1/x)} = \lim_{x\to\infty}\left(\frac{\log(x\mu(x))}{\log(x)} \right)^k = \left(1 + \lim_{x\to\infty}\frac{\log(\mu(x))}{\log(x)} \right)^k = 1.
\end{align*}
Thus we seek to arrange that $\lim_{x\to\infty}\log(\mu(x))/\log(x) = 0$; it is again convenient and natural to take $\mu$ such that $\mu(x) / \log^\beta x \to 1$ as $x\to\infty$ with $\beta>1$, which, as we have already seen, satisfies all the requirements of Theorem \ref{rapid_tails_theorem}. 
\end{example}
In many applications, particularly in economics and finance, it is important to understand the behaviour of systems driven or corrupted by random noise which is characterised by slow decay in the tails of the related distribution function --- so--called ``heavy tailed'' distributions. Our next example provides a simple application of our results in such a situation.

\begin{example}\label{examp.power}
We consider the case of a symmetric heavy tailed distribution with power law decay in the tails. 
Suppose that $H(n)$ are i.i.d. random variables such that there is $\alpha>0$ and finite $c_1,c_2>0$ for which
\[
\lim_{x\to -\infty}\frac{F(x)}{|x|^{-\alpha}}=c_1, \quad \lim_{ x\to+\infty}\frac{1-F(x)}{x^{-\alpha}}=c_2.  
\]
If $a_+$ and $a_-$ are sequences such that 
\[
\sum_{n=0}^\infty a_+(n)^{-\alpha} <+\infty, \quad \sum_{n=0}^\infty a_-(n)^{-\alpha} =+\infty,
\]
then $S(K,a_+)<+\infty$ for all $K>0$ and $S(K,a_-)=+\infty$ for all $K> 0$. Therefore, for all $K>0$, $\limsup_{n\to\infty} |H(n)|/a_+(n) \leq K$, on $\Omega_K^+$. The event $\Omega^+=\cap_{K\in \mathbb{Q}^+} \Omega_K^+$ has probability one and 
\[
\limsup_{n\to\infty} \frac{|H(n)|}{a_+(n)}=0,  \quad\text{on $\Omega^+$}.
\]
On the other hand, for all $K>0$ there is an a.s. event $\Omega_K^-$ such that
$\limsup_{n\to\infty} |H(n)|/a_-(n) \geq K$ on $\Omega_K^-$. Consider the event $\Omega^-=\cap_{K\in \mathbb{Z}^+} \Omega_K^-$. Then $\Omega^-$ is an almost sure event and we have 
\[
\limsup_{n\to\infty} \frac{|H(n)|}{a_-(n)}=+\infty,  \quad\text{on $\Omega^-$}.
\]
Finally, construct the a.s. event $\Omega^\ast=\Omega^+\cap\Omega^-$ and notice that
\[
\limsup_{n\to\infty} \frac{|H(n)|}{a_+(n)}=0,  \quad  \limsup_{n\to\infty} \frac{|H(n)|}{a_-(n)}=\infty, \text{ on $\Omega^\ast$}.
\] 
Applying part (a.) of Theorem~\ref{theorem.fluct} with $a=a_+$ and part (c.) with $a=a_-$, we obtain 
\[
\limsup_{n\to\infty} \frac{|x(n)|}{a_+(n)}=0,\quad \limsup_{n\to\infty} \frac{|x(n)|}{a_-(n)}=+\infty,\quad \text{on $\Omega^\ast$}.
\]
We now try to choose $a_+$ and $a_-$ ``close'' to one another, in an appropriate sense. For every $\epsilon>0$ sufficiently small take $a_{\pm}(n)$ to be $a_{\pm\epsilon}(n)=n^{1/\alpha\pm\epsilon}$. First, from the existence of the sequences $a_{\pm\epsilon}$ we conclude that there are a.s. events $\Omega_\epsilon^-$ and $\Omega_\epsilon^+$ such that
\[
\limsup_{n\to\infty} \frac{|x(n)|}{n^{1/\alpha -\epsilon}}=+\infty, \text{ on $\Omega_\epsilon^-$\,\,\, and}\quad \limsup_{n\to\infty} \frac{|x(n)|}{n^{1/\alpha +\epsilon}}=0, \text{ on $\Omega_\epsilon^+$}.
\]
Now we seek $\epsilon$--independent limits. We conclude from the limits above that 
\[
\limsup_{n\to\infty} \frac{\log |x(n)|}{\log n} \geq \frac{1}{\alpha}-\epsilon, \text{ on $\Omega_\epsilon^-$\,\,\, and}\quad \limsup_{n\to\infty} \frac{\log |x(n)|}{\log n} \leq \frac{1}{\alpha}+\epsilon, \text{ on $\Omega_\epsilon^+$}.
\]
Finally, by constructing the a.s. event $
\Omega^\ast=\{\cap_{\epsilon\in \mathbb{Q}^+} \Omega_\epsilon^+\} \cap \{\cap_{\epsilon\in \mathbb{Q}^+} \Omega_\epsilon^-\}$, it follows that
\[
\limsup_{n\to\infty} \frac{\log|x(n)|}{\log n}=\frac{1}{\alpha}, \quad\text{a.s.}
\]
\end{example}
While the preceding example is in some ways quite special, it is worthwhile pointing out one aspect which is generic. If the sequence of i.i.d. random variables $(H(n))_{n \geq 0}$ has a generalized power law decay in the tails, it is not possible to find an increasing sequence $(a(n))_{n \geq 0}$ tending to infinity which characterizes the partial maxima of $(H(n))_{n \geq 0}$ in the sense of \eqref{BC_exact}. Hence there was no possible ``smarter choice'' of the sequences $(a_\pm(n))_{n \geq 0}$ in Example \ref{examp.power} that could have given us information of the same quality as in Example \ref{examp.anormal}; this also demonstrates the practical merit of parts $(a.)$ and $(c.)$ of Theorem \ref{theorem.fluct}. The following result makes the claim above precise.
\begin{theorem}\label{rv_tails_theorem}
Let $(H(n))_{n \geq 0}$ be a sequence of i.i.d. random variables supported on $\mathbb{R}$ with continuous distribution function $F$. Suppose that one of the following holds:
\begin{enumerate}
\item[(i.)] $1-F \in \text{RV}_\infty(-\alpha)$ for some $\alpha > 0$ and $\lim_{x\to\infty}(1-F(x))/F(-x) = \infty$;
\item[(ii.)] $F \in \text{RV}_{-\infty}(-\alpha)$ for some $\alpha > 0$ and $\lim_{x\to\infty}(1-F(x))/F(-x) = 0$;
\item[(iii.)] $1-F \in \text{RV}_\infty(-\alpha)$ for some $\alpha > 0$ and $\lim_{x\to\infty}(1-F(x))/F(-x) = L \in (0,\infty)$.
\end{enumerate}
For each positive, increasing (deterministic) sequence $(a(n))_{n \geq 0}$ which tends to infinity, either
\begin{equation}\label{limsup_zero_or_infty}
\limsup_{n\to\infty}\frac{|H(n)|}{a(n)} = 0 \quad\mbox{a.s.,\,\,\, or}\quad\limsup_{n\to\infty}\frac{|H(n)|}{a(n)} = \infty \quad\mbox{a.s.} 
\end{equation}
\end{theorem}
We emphasise the fact that $a$ is a deterministic sequence in the result above because one could choose the a.s. increasing random sequence $a(n) = \max_{1\leq j \leq n}|H(j)|$ for each $n \geq 0$ and obtain nontrivial limits in \eqref{limsup_zero_or_infty}. However, understanding the stochastic process $H$ in terms of the closely related process $\left(\max_{1\leq j \leq n}|H(j)|\right)_{n \geq 0}$ clearly does not provide the same insight as a result such as Theorem \ref{rapid_tails_theorem}. The case when the distribution function has symmetric tails is trivially included in case $(iii.)$ of the result above and hence it applies to our earlier example of power law decay.
\section{Linearisation at Infinity}\label{linearisation}
Consider the nonlinear Volterra summation equation given by
\begin{equation} \label{eq.nonx}
x(n+1)=H(n+1)+\sum_{j=0}^n k(n-j)f(x(j)), \quad n\geq 1; \quad x(0)=\xi, 
\end{equation} 
where $f$ obeys
\begin{equation} \label{eq.f1}
f\in C(\mathbb{R};\mathbb{R}).
\end{equation}
Most models of macroeconomic growth have nonlinear equations of motion for one or more state variables; however, in many cases, these equations contain a term which is linear in the state. Furthermore, based on standard economic considerations, the nonlinear terms in these models are generally sublinear at infinity (in the sense that $\lim_{|x|\to\infty}g(x)/x=0$). For example, consider the seminal discrete time Solow model with no technological progress or population growth (see Solow \cite{solow1956contribution} or, for a more modern account, Acemoglu \cite{acemoglu2008introduction}). The evolution of output per capita $k$ is given by the nonlinear difference equation
\begin{equation}\label{solow}
k(n+1) = f(k(n)) := (1-\delta)k(n) + s\,g(k(n)), \quad n \geq 0,
\end{equation}
where $g$ obeys $\lim_{k \to \infty}g'(k)=0$ (due to the Inada conditions on the aggregate production function), $\delta\in(0,1)$ is the rate of depreciation per time period and $s\in(0,1)$ is the exogenous savings rate. The structure of the nonlinearity highlighted above  is, in some sense, generic and can be found in many of the macroeconomic growth models subsequently developed from the basic Solow model. The aforementioned linear leading order behaviour is also found in analogous macroeconomic models incorporating delays (see d'Halbis et al. \cite[Example 4.2]{d2014multiple} for a neoclassical growth model with delay). Hence it is natural, in this context, to assume that the nonlinear function $f$ in \eqref{eq.nonx} obeys
\begin{equation} \label{eq.f2}
\lim_{x\to\infty} \frac{f(x)}{x}=1, \quad \lim_{x\to-\infty} \frac{f(x)}{x}=1.
\end{equation}
Under the hypothesis \eqref{eq.f2}, it is reasonable to conjecture that $x$ may have similar asymptotic behaviour to its ``linearisation at infinity'', which solves the equation 
\begin{equation} \label{eq.y}
y(n+1)=H(n+1)+\sum_{j=0}^n k(n-j)y(j), \quad n\geq 1;\quad y(0)=\xi.
\end{equation}
The next result goes some distance to supporting this claim. It can then be used as a lemma to establish results on the fluctuation and growth of the solution to \eqref{eq.nonx}, once the growth or fluctuation of the sequence $H$ is sufficiently well--understood.
\begin{theorem} \label{theorem.nonlineargrowth}
Let $x$ and $y$ denote the solutions to \eqref{eq.nonx} and \eqref{eq.y} respectively. Suppose that $k$ and the solution $r$ of \eqref{eq.r} are summable, and that $f$ obeys \eqref{eq.f1} and \eqref{eq.f2}.  If $(a(n))_{n\geq 0}$ is an increasing sequence such that 
\[
\limsup_{n\to\infty} \frac{|H(n)|}{a(n)}<+\infty,
\]
then
\[
\limsup_{n\to\infty} \frac{|y(n)|}{a(n)}<+\infty,\quad\text{and}\quad\lim_{n\to\infty} \frac{x(n)-y(n)}{a(n)}=0.
\]
\end{theorem}
The result above shows that when $\Lambda_a|H|\in (0,\infty)$, not only is $\Lambda_a|y|\in (0,\infty)$, but also $\Lambda_a|x|=\Lambda_a|y|$; in other words, the solutions to \eqref{eq.nonx} and \eqref{eq.y} are coupled together so that the difference between them is always $o(a(n))$ as $n\to\infty$. Moreover, we have shown that $x(n)/a(n) \sim y(n)/a(n)$ as $n\to\infty$ and hence we can prove a corollary of Theorem~\ref{theorem.nonlineargrowth} which constitutes a nonlinear version of Theorem \ref{theorem.growth3}. 
\begin{theorem} \label{theorem.nongrowth2}
	Let $x$ denote the solution to \eqref{eq.nonx}, and suppose that $k$ and the solution $r$ of \eqref{eq.r} are summable. If $f$ obeys \eqref{eq.f1} and \eqref{eq.f2}, and $G_\lambda$ is  defined by \eqref{def.Glam}, then the following are equivalent:
	\begin{itemize}
		\item[(a.)] $H\in BG_{a,\lambda}$;
		\item[(b.)] $x\in BG_{a,\lambda}$.
	\end{itemize}
	Moreover, when $H \in BG_{a,\lambda}$,
	\begin{equation}\label{x_asym_relation_nonlinear}
	\frac{x(n)}{a(n)} \sim   (\lambda_a H)(n) + \sum_{j=1}^n r(j)\lambda^j (\lambda_a H)(n-j), \mbox{ as }n\to\infty,
	\end{equation}
	and similarly, when $x \in BG_{a,\lambda}$,
	\begin{equation}\label{H_asym_relation_nonlinear}
	\frac{H(n)}{a(n)} \sim (\lambda_a x)(n) - \sum_{j=0}^{n-1} k(j)\lambda^{j+1} (\lambda_a x)(n-j-1), \mbox{ as }n\to\infty.
	\end{equation} 
\end{theorem}
Similarly, it is straightforward to generate nonlinear versions of Theorem \ref{theorem.growth2}, Proposition \ref{prop.periodic}, and Proposition \ref{prop.ergodic} using the same line of arugment used to establish Theorem \ref{theorem.nongrowth2}.

Theorem~\ref{theorem.nonlineargrowth} has another nice corollary when $H$ obeys the hypotheses of Theorem~\ref{theorem.fluct}, and in fact the solution of the nonlinear equation inherits the fluctuation bounds seen in \eqref{eq.x}.
\begin{theorem} \label{theorem.nonlinearfluct}
	Let $x$ denote the solution to \eqref{eq.nonx}, and suppose that $k$ and the solution $r$ of \eqref{eq.r} are summable. Suppose further that $f$ obeys \eqref{eq.f1} and 
	\eqref{eq.f2}. If $(a(n))_{n\geq 0}$ is an increasing sequence such that $\Lambda_a|H|\in [0,\infty]$, then 
	\begin{itemize}
		\item[(a.)] $\Lambda_a|x|=0$ if and only if  $\Lambda_a|H|=0$. 
		\item[(b.)] $\Lambda_a|x|\in (0,\infty)$ if and only if $\Lambda_a|H|\in (0,\infty)$.
		\item[(c.)] $\Lambda_a|x|=\infty$ if and only if $\Lambda_a|H|=\infty$.
	\end{itemize}
\end{theorem}
The proof of the forward implications of (a.) and (b.) can be read off from Theorem~\ref{theorem.nonlineargrowth}. The reverse implications can be established by rewriting $H$ in terms of $x$ according to 
\begin{equation*} 
H(n+1)=x(n+1)-\sum_{j=0}^n k(n-j)f(x(j)), \quad n\geq 1,
\end{equation*}
and using \eqref{eq.f1} and \eqref{eq.f2}, as well as results concerning the growth of convolutions (see Lemma~\ref{lemma.convdiva}), to bound the right hand side. As to the proof of the last part, take as hypothesis that $\Lambda_a|H|=\infty$. Suppose now that $\Lambda_a|x|<\infty$; then applying part (b.) gives a contradiction, so $\Lambda_a|H|=\infty$ implies $\Lambda_a|x|=\infty$. If, on the other hand we take as hypothesis that 
$\Lambda_a|x|=\infty$, and suppose that $\Lambda_a|H|<\infty$, applying part (b.) again gives a contradiction. Hence $\Lambda_a|x|=\infty$ implies $\Lambda_a|H|=\infty$.
\section{Proof of Theorem~\ref{theorem.growth2}} 
First note that the proof of Theorem \ref{theorem.growth3} does not depend on the conclusion of Theorem \ref{theorem.growth2} whatsoever. If $H \in G_\lambda$, then $H \in BG_{a,\lambda}$ with $a = H$ and $\lambda_a H = 1$. Hence Theorem \ref{theorem.growth3} implies that
\begin{align}\label{thm.1.limit}
\lim_{n\to\infty}\frac{x(n)}{H(n)} = 1 + \sum_{j = 1}^\infty r(j)\lambda^j.
\end{align}
When $\lambda = 0$,  $\lim_{n\to\infty}x(n)/H(n) = 1$. If $\lambda \in (0,1]$, we claim that 
\[
\sum_{j=0}^\infty r(j)\lambda^j = \frac{1}{1 - \sum_{j = 0}^\infty k(j)\lambda^{j+1}},
\]
since $r \in \ell^1(\mathbb{Z}^+)$. To see this first note that both $\left( \sum_{j=0}^N r(j)\lambda^j \right)_{N \geq 0}$ and $\left( \sum_{j=0}^N k(j)\lambda^{j+1} \right)_{N \geq 0}$ are Cauchy sequences because $r \in \ell^1(\mathbb{Z}^+)$ and $k \in \ell^1(\mathbb{Z}^+)$. Furthermore, by the same considerations, the limits of these sequences are finite. Recall that 
\[
r(n) = \sum_{j=0}^{n-1}k(n-1-j)r(j), \quad n \geq 1, \quad r(0)=1.
\]
Suppose $N > 1$ and calculate as follows:
\begin{align*}
\sum_{n = 0}^N r(n)\lambda^n &= 1 + \sum_{n=1}^N \sum_{j = 0}^{n-1}k(n-1-j)r(j)\lambda^n = 1 + \sum_{n=1}^N \sum_{l=0}^{n-1} k(l)r(n-l-1)\lambda^{l+1}\lambda^{n-l-1} \\ &= 1 + \sum_{l=0}^N k(l)\lambda^{l+1} \left(\sum_{n=l+1}^N r(n-l-1)\lambda^{n-l-1}\right) = 1 + \sum_{l=0}^N k(l)\lambda^{l+1} \left(\sum_{m=0}^{N-l+1} r(m)\lambda^{m}\right).
\end{align*}
Hence
\[
\sum_{j = 0}^\infty r(j)\lambda^j = 1 + \left(\sum_{l = 0}^\infty k(l)\lambda^{l+1} \right) \left( \sum_{m = 0}^\infty r(m)\lambda^m \right), \quad \mbox{when }\lambda\in (0,1].
\]
Rearrange to obtain 
\[
\sum_{j = 0}^\infty r(j)\lambda^j = \frac{1}{1 - \sum_{l = 0}^\infty k(l)\lambda^{l+1}}, \quad \mbox{when }\lambda\in (0,1].
\]
Note that the right--hand side of the equality above is finite due to \eqref{eq.chareqn}. Combine the calculation above with \eqref{thm.1.limit} to obtain
\[
\lim_{n\to\infty}\frac{x(n)}{H(n)} = \frac{1}{1 - \sum_{l = 0}^\infty k(l)\lambda^{l+1}}, \quad \mbox{when }\lambda\in [0,1].
\]
The above limit clearly implies that $x \in G_\lambda$ also.

Similarly, if $x \in G_\lambda$, then $x \in BG_{a,\lambda}$ with $a=x$ and $\lambda_a x = 1$. By Theorem \ref{theorem.growth3} and equation \eqref{H_asym_relation}, we have
\[
\lim_{n\to\infty}\frac{H(n)}{x(n)} = 1 - \sum_{j=0}^\infty k(j)\lambda^{j+1}.
\]
By analogous considerations, the limit above can be used to prove the converse claim.
\section{Proof of Theorem \ref{theorem.growth3}}
First assume that $H \in BG_{a,\lambda}$ and show that \eqref{x_asym_relation} holds; it is clear that $x \in BG_{a,\lambda}$ follows from \eqref{x_asym_relation}. Since $H \in BG_{a,\lambda}$, there is a bounded sequence $((\lambda_a H)(n))_{n \geq 0}$ such that 
\[
\lim_{n\to\infty}\left|(\lambda_a H)(n)-\frac{H(n)}{a(n)}\right| =0,
\]
for some $(a(n))_{n \geq 0} \in G_\lambda$. From \eqref{eq.xrep}, we have that
\begin{equation}\label{eq.prelim.est}
\frac{x(n)}{a(n)}  =  \frac{H(n)}{a(n)} + \frac{r(n)x(0)}{a(n)} + \frac{1}{a(n)}\sum_{j=1}^{n-1}r(n-j)H(j), \quad n \geq 2.
\end{equation}
Hence
\begin{multline}\label{eq.est.converse}
\left|\frac{x(n)}{a(n)} - (\lambda_a H)(n) - \sum_{l=1}^n r(l)\lambda^l (\lambda_a H)(n-l)\right| \leq \left|\frac{H(n)}{a(n)} - (\lambda_a H)(n)\right| + \left|\frac{r(n)x(0)}{a(n)}\right|\\ + \left| \sum_{j=1}^{n-1}r(n-j)\frac{H(j)}{a(n)} - \sum_{l=1}^n r(l)\lambda^l (\lambda_a H)(n-l)\right|, \quad n\geq 2.
\end{multline}
The first two terms on the right--hand side above clearly tend to zero as $n\to\infty$ and thus it remains to show that the final term also has limit zero. Let $\epsilon>0$ be arbitrary in what follows. Since $H \in BG_{a,\lambda}$, there exists $N_1(\epsilon) > 2$ such that
\begin{align}\label{N_1_def_a}
\left|\frac{H(n)}{a(n)} - (\lambda_a H)(n)\right| < \epsilon, \mbox{ for all } n \geq N_1(\epsilon).
\end{align}
Given $N_1(\epsilon)$, since $a \in G_\lambda$ and $r \in \ell^1(\mathbb{Z}^+)$, there exists $N_2(\epsilon) > 1$ such that 
\begin{align}\label{N_2(epsilon)_defn}
\left| \frac{1}{a(n)}\sum_{j=1}^{N_1-1} r(n-j)H(j) \right| < \epsilon, \mbox{ for all }n\geq N_2(\epsilon)\mbox{ and }j\in\{1,\dots,N_3+1\}.
\end{align}
Similarly, because $r \in \ell^1(\mathbb{Z}^+)$, there exists $N_3(\epsilon)>1$ such that 
$\sum_{j=N_3}^\infty |r(j)| < \epsilon$. Finally, since $a \in G_\lambda$, there is an $N_4(\epsilon)>1$ such that 
\[
\left|\frac{a(n-j)}{a(n)} - \lambda^j \right| < \epsilon, \mbox{ for all } n \geq N_4(\epsilon).
\]
First concentrate  on the final term on the right--hand side of \eqref{eq.prelim.est} and decompose as follows:
\begin{multline*}
\sum_{j=1}^{n-1}r(n-j)\frac{H(j)}{a(n)} = \sum_{j=1}^{N_1-1}r(n-j)\frac{H(j)}{a(n)} + \frac{1}{a(n)}\sum_{j = N_1}^{n-1}r(n-j)a(j)\left\{ \frac{H(j)}{a(j)} - (\lambda_a H)(j)\right\}\\ + \frac{1}{a(n)}\sum_{j=N_1}^{n-1}r(n-j)a(j)(\lambda_a H)(j), \quad n \geq N_1+1.
\end{multline*}
Splitting the final term in the expression above, with $n \geq N_1 + N_3 +3$, then yields
\begin{multline}\label{eq.split.terms}
\sum_{j=1}^{n-1}r(n-j)\frac{H(j)}{a(n)} = \sum_{j=1}^{N_1-1}r(n-j)\frac{H(j)}{a(n)} + \frac{1}{a(n)}\sum_{j = N_1}^{n-1}r(n-j)a(j)\left\{ \frac{H(j)}{a(j)} - (\lambda_a H)(j)\right\}\\ + \frac{1}{a(n)}\sum_{j=N_1}^{n-N_3-2}r(n-j)a(j)(\lambda_a H)(j) + \frac{1}{a(n)}\sum_{l=1}^{N_3+1}r(l)a(n-l)(\lambda_a H)(n-l),
\end{multline}
where the order of summation was reversed in the final term. Subtract $\sum_{l=1}^n r(l)\lambda^l (\lambda_a H)(n-l)$ from the expression above and take absolute values to obtain
\begin{subequations}\label{eq.key.est.multiple}
\begin{align}
\left| \sum_{j=1}^{n-1}r(n-j)\frac{H(j)}{a(n)} - \sum_{l=1}^n r(l)\lambda^l (\lambda_a H)(n-l)\right| &\leq \left|\sum_{j=1}^{N_1-1}r(n-j)\frac{H(j)}{a(n)}\right| \label{est_a}\\ &+ \sum_{j = N_1}^{n-1}|r(n-j)|\left|\frac{a(j)}{a(n)}\right|\left| \frac{H(j)}{a(j)} - (\lambda_a H)(j)\right|\label{est_b}\\ &+ \sum_{j=N_1}^{n-N_3-2}|r(n-j)|\left|\frac{a(j)}{a(n)}\right||(\lambda_a H)(j)| \label{est_c}\\ &+ \sum_{l=1}^{N_3+1}|r(l)|\left|\frac{a(n-l)}{a(n)} - \lambda^l\right||(\lambda_a H)(n-l)| \label{est_d}\\ &+ \sum_{l = N_3+2}^n |r(l)|\lambda^l|(\lambda_a H)(n-l)|\label{est_e},
\end{align}
\end{subequations}
which is valid for each $n \geq N_1+N_3+3$. Now let $n \geq N_1+N_2+N_3+ N_4+1$; the term on the right--hand side of \eqref{est_a} is less than $\epsilon$ for all $n \geq N_2$ by construction (see \eqref{N_2(epsilon)_defn}). Estimate the right--hand side of \eqref{est_b} as follows:
\[
\sum_{j = N_1}^{n-1}|r(n-j)|\left|\frac{a(j)}{a(n)}\right|\left| \frac{H(j)}{a(j)} - (\lambda_a H)(j)\right| \leq \epsilon \,\sum_{j = N_1}^{n-1}|r(n-j)|\left|\frac{a(j)}{a(n)}\right| \leq \epsilon\,\bar{A}\,|r|_1,
\]
where we have used that $a(j)/a(n)$ can be uniformly bounded by $\bar{A}>0$ and that the summation index started at $N_1$ (see \eqref{N_1_def_a}). Next estimate the term in \eqref{est_c}:
\[
\sum_{j=N_1}^{n-N_3-2}|r(n-j)|\left|\frac{a(j)}{a(n)}\right||(\lambda_a H)(j)| \leq \bar{\Lambda}\,\bar{A}\sum_{j = N_1}^{n-N_3-2}|r(n-j)|\leq \bar{\Lambda}\,\bar{A}\sum_{l = N_3+2}^{\infty}|r(l)| \leq \epsilon\,\bar{\Lambda}\,\bar{A},
\]
where $\lambda_a H$ was uniformly bounded by $\bar{\Lambda}>0$ and we also used the definition of $N_3(\epsilon)$. Now note that $n-N_3-1 \geq N_4$ and hence that we may estimate from \eqref{est_d} as below:
\[
\sum_{l=1}^{N_3+1}|r(l)|\left|\frac{a(n-l)}{a(n)} - \lambda^l\right||(\lambda_a H)(n-l)| \leq \epsilon\,\sum_{l=1}^{N_3+1}|r(l)||(\lambda_a H)(n-l)| \leq \epsilon\,\Lambda\,|r|_1.
\]
The estimation of \eqref{est_e} is handled by simply noting that $\lambda \in [0,1]$, $r \in \ell^1(\mathbb{Z}^+)$, and that $\lambda_a H$ is bounded; thus
\[
\sum_{l = N_3+2}^n |r(l)|\lambda^l|(\lambda_a H)(n-l)| \leq \bar{\Lambda}\sum_{l = N_3+2}^\infty |r(l)| \leq \bar{\Lambda}\epsilon.
\]
Returning to \eqref{eq.key.est.multiple}, we have shown that 
\[
\left| \sum_{j=1}^{n-1}r(n-j)\frac{H(j)}{a(n)} - \sum_{l=1}^n r(l)\lambda^l (\lambda_a H)(n-l)\right| \leq \epsilon \left(1+\bar{A}\,|r|_1 + \bar{\Lambda}\,\bar{A} + \Lambda\,|r|_1 + \bar{\Lambda} \right), 
\]
for each $n \geq N_1 + N_2 + N_3 + N_4 + 1$. Letting $n\to\infty$ and then $\epsilon\to 0$ in the estimate above, we have proven that 
\[
\lim_{n\to\infty}\left| \sum_{j=1}^{n-1}r(n-j)\frac{H(j)}{a(n)} - \sum_{l=1}^n r(l)\lambda^l (\lambda_a H)(n-l)\right|=0, \mbox{ for each }\lambda \in [0,1],
\]
and hence demonstrated that \eqref{x_asym_relation} holds.

For the converse result assume $x \in BG_{a,\lambda}$ and note that 
\[
\frac{H(n)}{a(n)} = \frac{x(n)}{a(n)} - \frac{1}{a(n)}\sum_{j = 0}^{n-1}k(n-1-j)x(j) = \frac{x(n)}{a(n)} - k(n-1)x(0) - \frac{1}{a(n)}\sum_{j = 1}^{n-1}k(n-1-j)x(j), \quad n \geq 1.
\]
Hence, for $n \geq 1$, we have
\begin{multline}\label{eq.interim.est}
\left|\frac{H(n)}{a(n)} - (\lambda_a x)(n) + \sum_{j=0}^{n-1}k(j)\lambda^{j+1}(\lambda_a x)(n-j-1)\right| \leq \left|\frac{x(n)}{a(n)} - (\lambda_a x)(n)\right| + \left| \frac{k(n-1)x(0)}{a(n)}\right| \\ +\left|\sum_{j = 1}^{n-1}k(n-1-j)\frac{x(j)}{a(n)} - \sum_{j=0}^{n-1}k(j)\lambda^{j+1}(\lambda_a x)(n-j-1)\right|.
\end{multline}
Now rewrite the final sum on the right--hand side above as follows:
\[
\sum_{j=0}^{n-1}k(j)\lambda^{j+1}(\lambda_a x)(n-j-1) = \sum_{l=1}^{n}k(l-1)\lambda^{l}(\lambda_a x)(n-l).
\]
Substitute the expression above into \eqref{eq.interim.est} to obtain
\begin{multline*}
\left|\frac{H(n)}{a(n)} - (\lambda_a x)(n) + \sum_{j=0}^{n-1}k(j)\lambda^{j+1}(\lambda_a x)(n-j-1)\right| \leq \left|\frac{x(n)}{a(n)} - (\lambda_a x)(n)\right| + \left| \frac{k(n-1)x(0)}{a(n)}\right| \\ +\left|\sum_{j = 1}^{n-1}k(n-1-j)\frac{x(j)}{a(n)} - \sum_{l=1}^{n}k(l-1)\lambda^{l}(\lambda_a x)(n-l)\right|.
\end{multline*}
Compare the above estimate with \eqref{eq.est.converse}; these estimates are exactly analogous with $k(j-1)$ replaced by $r(j)$, $\lambda_a x$ replaced by $\lambda_a H$, and $x$ replaced by $H$. Repeat the argument above to complete the proof. 
\section{Proof of Proposition \ref{prop.periodic}}
We freely use elementary but nontrivial properties of almost periodic sequences throughout this argument and the reader is invited to consult Corduneanu \cite[Chapter 1]{corduneanu1989almost} for the requisite proofs. 

Assume $H \in PG_{a,\lambda}$. Since $H/a \in \text{AAP}(\mathbb{Z}^+)$, 
\[
\frac{H(n)}{a(n)} = \pi_H(n) + \phi(n),\quad  \text{for each }n \geq 0,
\]
for some $\pi_H \in AP(\mathbb{Z})$ and $(\phi(n))_{n \geq 0}$ such that $\phi(n)\to 0$ as $n\to\infty$. Moreover, $\pi_H$ is bounded (because it is almost periodic). Hence we may take $\lambda_a H = \pi_H$ in \eqref{def.BG_lam}, or in other words $PG_{a,\lambda} \subset BG_{a,\lambda}$. Therefore the asymptotic representation \eqref{x_asym_periodic} is valid by appealing to Theorem \ref{theorem.growth3}. Thus 
\begin{equation}\label{rep_1_periodic}
\frac{x(n)}{a(n)} = \pi_H(n) + \sum_{j=1}^n r(j)\lambda^j \pi_H(n-j) + \tilde{\phi}(n),\quad \text{for each }n \geq 0,
\end{equation}
where $\left(\tilde{\phi}(n)\right)_{n \geq 0}$ obeys $\tilde{\phi}(n)\to 0$ as $n\to\infty$. It is clear from \eqref{rep_1_periodic} that if $\lambda=0$, then $x/a \in AAP(\mathbb{Z}^+)$ and hence that $x \in PG_{a,\lambda}$. Henceforth assume that $\lambda\in (0,1]$. Consider the sequence $(\pi(n))_{n \geq 0}$ given by
\[
\pi(n) = \sum_{j=n+1}^\infty r(j)\lambda^{j+1}\pi_H(n-j), \quad \text{for each }n\geq 0,
\]
and note that it is well defined because $\pi_H(n)$ is bounded and defined for all $n \in \mathbb{Z}$. Furthermore, since $r \in \ell^1(\mathbb{Z}^+)$ and $\pi_H$ is bounded, $\lim_{n\to\infty}\pi(n)=0$. Thus adding $\pi(n)$ to both sides of \eqref{rep_1_periodic} yields
\begin{equation}\label{rep_2_periodic}
\frac{x(n)}{a(n)} = \sum_{j=0}^\infty r(j)\lambda^j \pi_H(n-j) + \Phi(n),\quad \text{for each }n \geq 0,
\end{equation}
where $\Phi(n) = \tilde{\phi}(n) - \pi(n)$ for each $n \geq 0$ and hence $\Phi(n)\to 0$ as $n\to\infty$. We claim that the sequence $(\pi_x(n))_{n \in \mathbb{Z}}$ given by
\[
\pi_x(n) = \sum_{j=0}^\infty r(j)\lambda^j \pi_H(n-j),\quad n \in \mathbb{Z},
\]
is almost periodic. To see this we require the definition of a normal sequence, which we now state:
\begin{definition}
$(\pi_H(n))_{n \in \mathbb{Z}}$ is normal if for any sequence $(\alpha(l))_{l \in \mathbb{Z}} \subset \mathbb{Z}$ there exists a subsequence $(\alpha'(l))_{l \in \mathbb{Z}}$, a sequence $(\bar{\pi}_H(n))_{n \in \mathbb{Z}}$, and an integer $L(\epsilon)$ such that 
\begin{equation}\label{UC_periodic}
\left|\pi_H\left(n+ \alpha'(l)\right) - \bar{\pi}_H(n)\right| < \epsilon \quad \text{for }l \geq L(\epsilon) \text{ and each } n \in \mathbb{Z}.
\end{equation}
In other words, $\pi_H(n+ \alpha'(l))$ converges uniformly with respect to $n \in \mathbb{Z}$ as $l \to \infty$.
\end{definition}
A sequence is almost periodic if and only if it is normal. In order to show that $\pi_x$ is normal, let $(\alpha(l))_{l \in \mathbb{Z}} \subset \mathbb{Z}$ be an arbitrary sequence. Since $\pi_H$ is normal, there exists a subsequence $(\alpha'(l))_{l \in \mathbb{Z}}$, a sequence $(\bar{\pi}_H(n))_{n \in \mathbb{Z}}$, and an integer $L(\epsilon)$ such that \eqref{UC_periodic} holds for $\pi_H$. Define the sequence $\left( \bar{\pi}_x(n) \right)_{n \in \mathbb{Z}}$ by
\[
\bar{\pi}_x(n) = \sum_{j=0}^\infty r(j)\lambda^j \bar{\pi}_H(n-j),\quad n \in \mathbb{Z}.
\]   
Suppose $l \geq L(\epsilon)$ and estimate as follows
\begin{align*}
|\pi_x(n+ \alpha'(l)) - \bar{\pi}_x(n)| \leq \sum_{j=0}^\infty |r(j)|\,|\pi_H(n - j + \alpha'(l)) - \bar{\pi}_H(n-j)| < \epsilon \,|r|_1, \quad \text{for each }n \in \mathbb{Z}.
\end{align*}
Hence $\pi_x \in AP(\mathbb{Z})$ and, by \eqref{rep_2_periodic}, $x/a \in AAP(\mathbb{Z}^+)$. Therefore $x \in PG_{a,\lambda}$, as claimed.

For the converse result, assume that $x \in PG_{a,\lambda}$, so that there exists $a \in G_\lambda$ such that $x/a \in AAP(\mathbb{Z}^+)$. As before, we can apply Theorem \ref{theorem.growth3} to show that
\begin{equation}\label{H/a_rep}
\frac{H(n)}{a(n)} = \pi_x(n) - \sum_{j=0}^{n-1}k(j)\lambda^{j+1}\pi_x(n-j-1) + \phi(n), \quad \text{for each }n \geq 0,
\end{equation}
where $\pi_x \in AP(\mathbb{Z})$ and $(\phi(n))_{n \geq 0}$ obeys $\phi(n) \to 0$ as $n\to\infty$. Once more note that if $\lambda=0$, then $H/a \in AAP(\mathbb{Z}^+)$ trivially. Assume henceforth that $\lambda \in (0,1]$ and rewrite \eqref{H/a_rep} as follows
\begin{align*}
\frac{H(n)}{a(n)} = \pi_x(n) - \sum_{l=1}^{n}k(l-1)\lambda^{l}\pi_x(n-l) + \phi(n)
= \sum_{l=0}^{n}\tilde{r}(l)\lambda^{l}\pi_x(n-l) + \phi(n), \quad \text{for each }n \geq 0,
\end{align*}
where $\tilde{r}(l) = -k(l-1)$ for each $l \geq 0$ and $k(-1) = -1$. Thus $\tilde{r} \in \ell^1(\mathbb{Z}^+)$ with $|\tilde{r}|_1 = 1+ |k|_1$. The representation above is exactly analogous to that of \eqref{rep_1_periodic} and the proof proceeds as in the previous case (with $\tilde{r}$ in the role of $r$ and $\pi_x$ in place of $\pi_H$).
\section{Proof of Proposition \ref{prop.ergodic}}
Suppose $H \in AG_{a,\lambda}$ with $\lim_{n\to\infty}(\mu_a H)(n) = {\mu_a H}^*$ and note that $AG_{a,\lambda} \subset BG_{a,\lambda}$. By Theorem \ref{theorem.growth3}, we may write
\begin{equation}\label{x_asym_ergodic}
\frac{x(n)}{a(n)} = (\lambda_a H)(n) + \sum_{j=1}^n r(j)\lambda^j (\lambda_a H)(n-j) + R(n), \quad n \geq 1,
\end{equation}
where $(R(n))_{n \geq 0}$ obeys $\lim_{n\to\infty}R(n)=0$. Note that if $\lambda = 0$ we may trivially conclude that
\[
\lim_{n\to\infty}\frac{1}{n}\sum_{j=1}^n \frac{x(j)}{a(j)} = {\mu_a H}^*,
\]
as claimed. Assume henceforth that $\lambda \in (0,1]$. Thus
\begin{equation}\label{eq.sum.initial}
\frac{1}{n}\sum_{j=1}^n \frac{x(j)}{a(j)} = \frac{1}{n}\sum_{j=1}^n \sum_{l = 0}^j r(l) \lambda^l (\lambda_a H)(j-l) + \frac{1}{n}\sum_{j=1}^n R(j), \quad n \geq 1.
\end{equation}
Begin by rewriting the first term on the right--hand side of \eqref{eq.sum.initial} as follows:
\begin{align*}
\frac{1}{n}\sum_{j=1}^n \sum_{l = 0}^j r(l) \lambda^l (\lambda_a H)(j-l) &= \frac{1}{n}\sum_{k=0}^n \sum_{j = 1 \vee k}^n r(j-k) \lambda^{j-k} (\lambda_a H)(k)\\ &= \frac{1}{n}\sum_{j=1}^n r(j) \lambda^j (\lambda_a H)(0) + \frac{1}{n}\sum_{k=1}^n \sum_{j = 1}^n r(j-k) \lambda^{j-k} (\lambda_a H)(k) \\ &= \frac{1}{n}\sum_{j=1}^n r(j) \lambda^j (\lambda_a H)(0) + \frac{1}{n}\sum_{k=1}^n \sum_{i=0}^{n-k} r(i) \lambda^{i}(\lambda_a H)(k)\\ &= \frac{1}{n}\sum_{j=1}^n r(j) \lambda^j (\lambda_a H)(0) + \frac{1}{n}\sum_{k=1}^n \rho(n-k) (\lambda_a H)(k), \quad n \geq 1,
\end{align*}
where $\rho(n) = \sum_{i = 0}^n r(i)\lambda^i$ for each $n \geq 0$. Substitute the above expression into \eqref{eq.sum.initial} to obtain
\begin{equation}\label{eq.exact.formula.key}
\frac{1}{n}\sum_{j=1}^n \frac{x(j)}{a(j)} = \frac{1}{n}\sum_{j=1}^n R(j) + \frac{1}{n}\sum_{j=1}^n r(j) \lambda^j (\lambda_a H)(0) + \frac{1}{n}\sum_{k=1}^n \rho(n-k) (\lambda_a H)(k), \quad n \geq 1.
\end{equation}
Since $r \in \ell^1(\mathbb{Z}^+)$, recall from the proof of Theorem \ref{theorem.growth2} that
\[
\lim_{n\to\infty}\rho(n) = \sum_{i = 0}^\infty r(i)\lambda^i = \frac{1}{1 - \sum_{i = 0}^\infty k(i)\lambda^{i+1}} =: \rho^*.
\]
By similar considerations, the first two terms on the right--hand side of \eqref{eq.exact.formula.key} tend to zero as $n\to\infty$. We claim that the third term also tends to a limit, namely
\[
\lim_{n\to\infty}\frac{1}{n}\sum_{k=1}^n \rho(n-k) (\lambda_a H)(k) = \frac{{\mu_a H}^*}{1 - \sum_{i = 0}^\infty k(i)\lambda^{i+1}}.
\]
We prove the limit above for ${\mu_a H}^*$ finite but our conclusions are also true when ${\mu_a H}^* = \pm \infty$, interpreting the relevant formulae correctly. Suppose ${\mu_a H}^*$ is finite and consider
\[
\left|\frac{1}{n}\sum_{k=1}^n\rho(n-k)(\lambda_a H)(k) - \frac{1}{n}\sum_{k=1}^n\rho^* (\lambda_a H)(k)\right| = \left|\frac{1}{n}\sum_{k=1}^n\{\rho(n-k) - \rho^*\}(\lambda_a H)(k)\right|, \quad n \geq 1.
\]
Now, because $\lim_{n\to\infty}\rho(n)=\rho^*$, there exists $N(\epsilon)>1$ such that for all $n \geq N(\epsilon)$ we have $|\rho(n)-\rho^*|< \epsilon$, for an arbitrary $\epsilon>0$. Define $\bar{\rho} = \sup_{n \in \mathbb{Z}^+}|\rho(n)-\rho^*|$ and $\bar{S}_H = \sup_{n \in \mathbb{Z}^+}|(\lambda_a H)(n)|$, recalling that $\lambda_a H$ is a bounded sequence. Thus
\begin{align*}
\left|\frac{1}{n}\sum_{k=1}^n\{\rho(n-k) - \rho^*\} (\lambda_a H)(k)\right| &\leq \left|\frac{1}{n}\sum_{k=1}^{n-N}\{\rho(n-k) - \rho^*\} (\lambda_a H)(k)\right| \\ & \qquad+ \left|\frac{1}{n}\sum_{k=n-N+1}^{n}\{\rho(n-k) - \rho^*\} (\lambda_a H)(k)\right| \\
&\leq \frac{\epsilon}{n}\sum_{k=1}^{n-N} |(\lambda_a H)(k)| + \left|\frac{1}{n}\sum_{k=n-N+1}^{n}\{\rho(n-k) - \rho^*\} (\lambda_a H)(k)\right|\\ &\leq \frac{\epsilon(n-N)\bar{S}_H}{n} + \frac{\bar{\rho}\, \bar{S}_H\,N}{n},
\end{align*}
for each $n \geq N(\epsilon)+1$. Since it is clear that the right--hand side of the inequality above can be made arbitrarily small for $n$ sufficiently large we have proven that
\[
\lim_{n\to\infty}\left|\frac{1}{n}\sum_{k=1}^n\rho(n-k)(\lambda_a H)(k) - \frac{1}{n}\sum_{k=1}^n\rho^* (\lambda_a H)(k)\right| = 0,
\]
and therefore
\[
\lim_{n\to\infty}\frac{1}{n}\sum_{k=1}^n\rho(n-k)(\lambda_a H)(k) = \lim_{n\to\infty}\frac{1}{n}\sum_{k=1}^n\rho^* (\lambda_a H)(k) = \rho^* {\mu_a H}^* = \frac{\mu_H^*}{1-\sum_{i=0}^\infty \lambda^{i+1}k(i)},
\]
as claimed. Using the limit above and sending $n\to\infty$ in \eqref{eq.exact.formula.key}  yields the desired conclusion.

Conversely, assume $x \in AG_{a,\lambda}$. By Theorem \ref{theorem.growth3}, we can write
\begin{equation}\label{H_asym_ergodic}
\frac{H(n)}{a(n)} = (\lambda_a x)(n) - \sum_{j=0}^{n-1} k(j)\lambda^{j+1} (\lambda_a x)(n-j-1)+ R(n), \quad n \geq 1,
\end{equation}
where $(R(n))_{n \geq 0}$ obeys $\lim_{n\to\infty}R(n)=0$. Once more note that the case $\lambda=0$ is trivial and assume that $\lambda \in (0,1]$. Define $k(-1) = -1$, so that 
\begin{align*}
\frac{H(n)}{a(n)} &= - \sum_{j=-1}^{n-1} k(j)\lambda^{j+1} (\lambda_a x)(n-j-1)+ R(n)= - \sum_{l=0}^{n} k(l-1)\lambda^{l} (\lambda_a x)(n-l)+ R(n)\\
&= \sum_{l=0}^{n} \tilde{r}(l)\lambda^{l} (\lambda_a x)(n-l)+ R(n), \quad n \geq 1,
\end{align*}
where $\tilde{r}(l) = -k(l-1)$ for each $l \geq 0$. At this point we have a formula exactly analogous to equation \eqref{x_asym_ergodic}. Repeating the argument from the previous case we arrive at the analogue of equation \eqref{eq.exact.formula.key}, namely
\begin{equation}\label{eq.keyidentity.converse}
\frac{1}{n}\sum_{j=1}^n \frac{H(j)}{a(j)} = \frac{1}{n}\sum_{j=1}^n R(j) + \frac{1}{n}\sum_{j=1}^n \tilde{r}(j)\lambda^j (\lambda_a x)(0) + \frac{1}{n}\sum_{j=1}^n \tilde{\rho}(n-j)(\lambda_a x)(j),
\end{equation}
where $\tilde{\rho}(n) = \sum_{i = 0}^n \tilde{r}(i)\lambda^i$ for $n \geq 0$. Now notice that
\begin{align*}
\tilde{\rho}(n) = \sum_{i = 0}^n \tilde{r}(i)\lambda^i = -\sum_{i = 0}^n k(i-1)\lambda^i = 1-\sum_{i = 1}^n k(i-1)\lambda^i = 1-\sum_{j = 0}^{n-1} k(j)\lambda^{j+1}, \quad n \geq 1.
\end{align*}
Hence $\lim_{n\to\infty}\tilde{\rho}(n) = 1 - \sum_{j = 0}^\infty k(j)\lambda^{j+1}$. By the same argument as before, we have that
\[
\lim_{n\to\infty}\frac{1}{n}\sum_{j=1}^n \tilde{\rho}(n-j)(\lambda_a x)(j) = {\mu_a x}^*\left( 1 - \sum_{j = 0}^\infty k(j)\lambda^{j+1} \right).
\]
Therefore, by sending $n\to\infty$ in \eqref{eq.keyidentity.converse}, we have
\[
\lim_{n\to\infty}\frac{1}{n}\sum_{j=1}^n \frac{H(j)}{a(j)} = {\mu_a x}^*\left( 1 - \sum_{j = 0}^\infty k(j)\lambda^{j+1} \right),
\]
as required.
\section{Proof of Theorem~\ref{theorem.fluct}}
\subsection{A bound on convolution growth}
We note that results very similar to Lemma \ref{lemma.convdiva} and Theorem \ref{theorem.fluct} are part of the well--established theory in the area of Volterra equations and hence the following results are in some sense ``known''. Nonetheless, providing our own proofs and formulations is most convenient from a presentational viewpoint. Furthermore, our interest in stochastic equations (cf. Section \ref{sec.examples}) and linearisation (cf. Section \ref{linearisation}) strongly motivates both the results of this section and our presentational emphasis on unifying the cases when $\Lambda_a |H|$ (resp. $\Lambda_a |x|$) is zero, finite, or infinite. 

We first prove a preliminary lemma.
\begin{lemma} \label{lemma.convdiva}
Suppose that $a$ is an increasing sequence with $a(n)\to\infty$ as $n\to\infty$ and that $k$ is summable. If $\Lambda_a|H|\in [0,\infty)$, then 
\begin{equation} \label{eq.kasthdiva}
\limsup_{n\to\infty} \left|\frac{1}{a(n)}\sum_{j=1}^n k(n-j) H(j)\right| \leq \sum_{j=0}^\infty |k(j)| \cdot \Lambda_a|H|.
\end{equation}
\end{lemma}
\begin{proof}
Since $\Lambda_a|H|\in [0,\infty)$, it follows for every $\epsilon>0$ that there is $N(\epsilon)>0$ such that 
\[
|H(n)|\leq (\epsilon + \Lambda_a|H|)a(n), \quad n\geq N(\epsilon).
\]
Suppose $n\geq N(\epsilon)$ and estimate as follows:
\begin{align*}
\left|\frac{1}{a(n)}\sum_{j=1}^n k(n-j) H(j)\right| &\leq \frac{1}{a(n)}\sum_{j=1}^{N-1} |k(n-j)| |H(j)| + \frac{1}{a(n)}\sum_{j=N}^{n} |k(n-j)| |H(j)| \\
&\leq \frac{1}{a(n)}\sum_{j=1}^{N-1} |k(n-j)| \cdot \sup_{1\leq j\leq N-1}|H(j)| 
+ (\epsilon + \Lambda_a|H|) \sum_{j=N}^{n} |k(n-j)| \frac{a(j)}{a(n)} \\
&\leq \frac{1}{a(n)}\sum_{l=n-(N-1)}^{n-1} |k(l)|  \sup_{1\leq j\leq N-1}|H(j)|  + (\epsilon + \Lambda_a|H|) \sum_{l=0}^{n-N} |k(l)|. 
\end{align*}
Since $k$ is summable and $a(n)\to\infty$ as $n\to\infty$, we have 
\[
\limsup_{n\to\infty} \left|\frac{1}{a(n)}\sum_{j=1}^n k(n-j) H(j)\right|
\leq (\epsilon + \Lambda_a|H|) \sum_{l=0}^\infty |k(l)|.
\]
Since $\epsilon>0$ is arbitrary, letting $\epsilon\to 0^+$ yields the result.
\end{proof}

\subsection{Proof of Theorem~\ref{theorem.fluct}}
Suppose that $\Lambda_a|H|\in [0,\infty)$. 
Since $r$ is summable and $\Lambda_a|H|\in [0,\infty)$, we have from Lemma~\ref{lemma.convdiva} that 
\[
\limsup_{n\to\infty} \left|\frac{1}{a(n)}\sum_{j=1}^n r(n-j) H(j)\right| \leq \sum_{j=0}^\infty |r(j)| \cdot \Lambda_a|H|.
\]
Since $r$ is summable and $a(n)\to \infty$ as $n\to\infty$, from the representation \eqref{eq.xrep} we get 
\begin{equation*} 
\limsup_{n\to\infty} \frac{|x(n)|}{a(n)}\leq \sum_{j=0}^\infty |r(j)| \cdot \Lambda_a|H|.
\end{equation*}
Hence 
\begin{equation} \label{eq.Lxupper}
\Lambda_a|x|\leq  \sum_{j=0}^\infty |r(j)| \cdot \Lambda_a|H|.
\end{equation}
Suppose on the other hand that $\Lambda_a|x|\in [0,\infty)$. Rearranging \eqref{eq.x} yields
\[
H(n+1)=x(n+1)-\sum_{j=0}^n k(n-j) x(j) = x(n+1)- k(n)x(0) - \sum_{j=1}^n k(n-j) x(j).
\]
By Lemma~\ref{lemma.convdiva} we have 
\[
\limsup_{n\to\infty} \left|\frac{1}{a(n)}\sum_{j=1}^n k(n-j) x(j)\right|\leq \sum_{j=0}^\infty |k(j)| \cdot  \Lambda_a|x|.
\]
Since $a$ is increasing and $k(n)\to 0$ as $n\to\infty$, we have 
\begin{align*}
\limsup_{n\to\infty} \frac{|H(n+1)|}{a(n+1)}  &\leq 
\limsup_{n\to\infty} \left\{\frac{|x(n+1)|}{a(n+1)} +  \frac{|k(n)|}{a(n+1)}|x(0)| 
+\frac{1}{a(n)}\left|\sum_{j=1}^n k(n-j) x(j)\right| \cdot \frac{a(n)}{a(n+1)}\right\}\\
&\leq \Lambda_a|x| +\sum_{j=0}^\infty |k(j)| \cdot  \Lambda_a|x|.
\end{align*}
Therefore
\begin{equation} \label{eq.LHupper}
\Lambda_a|H|\leq \Lambda_a|x| \left(1 +\sum_{j=0}^\infty |k(j)| \right). 
\end{equation}

To prove part (a.) of the result, suppose that $\Lambda_a|H|=0$. By \eqref{eq.Lxupper}, we have that $\Lambda_a|x|=0$. On the other hand, if  
$\Lambda_a|x|=0$, by \eqref{eq.LHupper}, we have $\Lambda_a|H|=0$. 

To prove part (b.), we start by showing that $\Lambda_a|H|\in (0,\infty)$ implies $\Lambda_a|x|\in (0,\infty)$. Suppose therefore that 
$\Lambda_a|H|\in (0,\infty)$. Then by \eqref{eq.Lxupper}, we have that $\Lambda_a|x|\in [0,\infty)$. Suppose that $\Lambda_a|x|=0$. Then by part (a.), we have $\Lambda_a|H|=0$, a contradiction. Thus we must have $\Lambda_a|x|\in (0,\infty)$. 

To prove the converse statement, we suppose now that $\Lambda_a|x|\in (0,\infty)$. By \eqref{eq.LHupper} it follows that 
$\Lambda_a|H|\in [0,\infty)$. If we assume that $\Lambda_a|H|=0$, then by part (a.), we have that $\Lambda_a|x|=0$, which gives a contradiction. Therefore we must have $\Lambda_a|H|\in (0,\infty)$.

To prove part (c.), we start by showing that $\Lambda_a|H|=+\infty$ implies $\Lambda_a|x|=+\infty$. Suppose not, so that $\Lambda_a|x|\in [0,\infty)$. Then the argument used to deduce \eqref{eq.LHupper} is valid and we have that $\Lambda_a|H|<+\infty$, which is a contradiction. To prove the reverse implication, we have by hypothesis that $\Lambda_a|x|=+\infty$. Suppose now that 
$\Lambda_a|H|<+\infty$. Then the argument used to prove \eqref{eq.Lxupper} is valid, and we have that $\Lambda_a|x|<+\infty$, which 
is a contradiction.
\section{Proof of Theorem \ref{thm.phi.moments}}
Estimating from \eqref{eq.xrep} we have 
\begin{align}\label{eq.phi.est.1}
|x(n)| \leq \sum_{j=0}^n |r(n-j)| |H(j)| \leq |r|_1 \sum_{j=0}^n \frac{|r(n-j)|}{\sum_{l=0}^n |r(l)|} |H(j)|, \quad n \geq 1.
\end{align}
Now apply $\phi$ to the expression above and use Jensen's inequality to obtain
\[
\phi(|x(n)|) \leq  \phi\left(|r|_1 \sum_{j=0}^n \frac{|r(n-j)|}{\sum_{l=0}^n |r(l)|} |H(j)|\right) \leq \frac{1}{\sum_{l=0}^n |r(l)|} \sum_{j=0}^n |r(n-j)|  \phi\left(|r|_1 |H(j)|\right), \quad n \geq 1.
\]
Since $r \in \ell^1(\mathbb{Z}^+)$, there exists an $N_1(\epsilon)>1$ such that $\sum_{l=0}^n |r(l)| > (1-\epsilon) |r|_1$ for all $n \geq N_1(\epsilon)$ and hence $1/\sum_{l=0}^{N_1} |r(l)| < 1/(1-\epsilon)|r|_1$ for all $n \geq N_1(\epsilon)$, with $\epsilon \in (0,1)$ arbitrary. Returning to \eqref{eq.phi.est.1}, we have that
\[
\phi(|x(n)|) \leq \frac{1}{(1-\epsilon) |r|_1} \sum_{j=0}^n |r(n-j)|  \phi\left(|r|_1 |H(j)|\right), \quad n \geq N_1(\epsilon).
\]
With $N$ sufficiently large, summing over the previous inequality yields
\begin{align*}
\sum_{n = N_1}^N \phi(|x(n)|) &\leq \frac{1}{(1-\epsilon) |r|_1}\sum_{n = N_1}^N  \sum_{j=0}^n |r(n-j)|  \phi\left(|r|_1 |H(j)|\right)\\ &= \frac{1}{(1-\epsilon) |r|_1}\sum_{j=0}^N  \left\{\sum_{n=N_1 \wedge j}^N |r(n-j)|\right\}  \phi\left(|r|_1 |H(j)|\right)\leq \frac{1}{1-\epsilon}\sum_{j=0}^N \phi\left(|r|_1 |H(j)|\right).
\end{align*}
Adding the remaining terms to the sums on the left--hand side of the above inequality, we have
\begin{align*}
\sum_{n = 0}^N \phi(|x(n)|) &\leq \frac{1}{1-\epsilon}\sum_{j=0}^N \phi\left(|r|_1 |H(j)|\right) + \sum_{n=0}^{N_1-1} \phi(|x(n)|).
\end{align*}
Therefore 
\[
\limsup_{N\to \infty}\frac{1}{N}\sum_{n = 0}^N \phi(|x(n)|) \leq \frac{1}{1-\epsilon}\limsup_{N\to \infty}\frac{1}{N}\sum_{j=0}^N \phi\left(|r|_1 |H(j)|\right),
\]
and letting $\epsilon\to 0^+ $ gives the desired conclusion. 

The second claim is proven analogously; first note that we can write
\[
H(n) = \sum_{j=0}^n \rho(n-j)x(j), \quad n \geq 1,
\]
where $\rho(j) = -k(j+1)$, $\rho(0)=1$ and $|\rho|_1 = 1 + |k|_1$. Now apply the argument above with $\rho$ in place of $r$ to complete the proof.
\section{Justification of Example \ref{eg.non_unit_multiplier}}
In this section we use the standard probabilistic notation $x \overset{d}{\sim} N(\mu,\sigma^2)$ to denote the random variable $x$ having a normal distribution with mean $\mu$ and variance $\sigma^2$.

We cannot proceed by direct calculation due to the nonstationarity of $(x(n))_{n \geq 0}$. Instead, by extending the filtration in a suitable manner, write
\begin{equation}\label{eq.rep.extended}
x(n) = x^*(n) + R(n), \quad\mbox{for each }n \geq 1,
\end{equation}
where $x^*(n) = \sum_{j = -\infty}^n r(n-j)H(j)$ for $n \geq 0$ and $R(n) = r(n)x(0) - \sum_{j = -\infty}^0 r(n-j)H(j)$ for $n \geq 0$. Since $r \in \ell^1(\mathbb{Z}^+) \subset \ell^2(\mathbb{Z}^+)$, we can use the dominated convergence theorem to show that 
\[
\mathbb{E} [ x^*(n) ] = 0, \quad \text{Var}[x^*(n)] = \mathbb{E} [ (x^*)^2(n) ] = \sigma^2 \sum_{l=0}^\infty r^2(l), \quad \mbox{for each }n \geq 0.
\]
In fact, $(x^*(n))_{n \geq 0}$ is a strongly stationary sequence, since $x^*(n) \overset{d}{\sim} N\left(0,\sigma^2 \sum_{l=0}^\infty r^2(l)\right)$ for each $n \geq 0$. Hence, by Birkhoff's ergodic theorem,
\[
\lim_{n\to\infty}\frac{1}{n}\sum_{j = 0}^n x^*(j) = 0 \quad \text{a.s.,}\quad\text{and}\quad\lim_{n\to\infty}\frac{1}{n}\sum_{j = 0}^n (x^*)^2(j) = \sigma^2 \sum_{l=0}^\infty r^2(l) \quad \text{a.s.} 
\]
From \eqref{eq.rep.extended}, we have
\[
\frac{1}{n}\sum_{j = 0}^n x^2(j) = \frac{1}{n}\sum_{j=0}^n (x^*)^2(j) + \frac{2}{n}\sum_{j=0}^n x^*(j)R(j) + \frac{1}{n}\sum_{j=0}^n R^2(j), \quad n \geq 1.
\]
Thus if we can show that $\lim_{n\to\infty}R(n)= 0$ a.s. we will have proven that
\[
\lim_{n\to\infty}\frac{1}{n}\sum_{j = 0}^n x^2(j) = \sigma^2 \sum_{l=0}^\infty r^2(l) = \sigma^2 (|r|_2)^2 \quad a.s.
\] 
Define $R_1(n) = \sum_{j=-\infty}^0 r(n-j)H(j)$ for each $n \geq 0$. Hence
\[
R^2(n) = r^2(n)x^2(0) - 2r(n)x(0)R_1(n) + R_1^2(n), \quad n \geq 0.
\]
From the equality above, we see that $\lim_{n\to\infty}R_1(n) = 0$ a.s. would imply that $\lim_{n\to\infty}R(n) = 0$ a.s., since $r(n)\to 0$ as $n\to \infty$. Hence it remains to prove that $\lim_{n\to\infty}R_1(n) = 0$ a.s. The dominated convergence theorem can be used to show that 
\[
\mathbb{E}[R_1(n)] = 0, \quad \mathbb{E}[R_1^2(n)] = \sigma^2 \sum_{l = n}^\infty r^2(l), \quad n \geq 0.
\]
In fact, $R_1(n) \overset{d}{\sim} N\left(0,\sigma^2 \sum_{l = n}^\infty r^2(l)\right)$ for each $n \geq 0$. To see this, fix $n \geq 0$ and define the sequence $R_1^N(n) = \sum_{j = -N}^0 r(n-j)$ for $N \geq 0$. Clearly, $R_1^N(n) \to R_1(n)$ a.s. as $N \to \infty$ and $R_1^N(n) \overset{d}{\sim} N\left(0,\sigma^2 \sum_{l = n}^{n+N} r^2(l)\right)$ for each $N \geq 0$. By explicitly writing down the characteristic function we can see that $\lim_{N\to\infty}R_1^N(n)=R_1(n)$ is normal with mean zero and variance $\sigma^2 \sum_{l = n}^\infty r^2(l)$ for each $n \geq 0$.

By a standard Borel--Cantelli argument the side condition $\log(n)\sum_{l=n}^\infty r^2(l) \to 0$ as $n\to\infty$ can then be used to show that $R_1(n)\to 0$ a.s. as $n\to\infty$, completing the proof (alternatively this can be deduced from Example \ref{examp.anormal} with an appropriate choice of $(a(n))_{n \geq 0}$ and constant $K$).
\section{Proof of Theorem \ref{rapid_tails_theorem}}
From \eqref{S_defn}, we have
\[
S(a,K) = \sum_{n=1}^\infty F(-K a(n)) + G(K a(n)),
\]
where $K>0$ is a constant, $a$ is a positive increasing sequence which tends to infinity, and $G(x) = 1 - F(x)$ for each $x \in \mathbb{R}$. Choose $K = 1$ and note that the symmetry of the distribution function means that the finiteness of $S(a,1)$ is equivalent to the finiteness of $\sum_{n=1}^\infty G(a(n))$. 

For each $n > 1$, take $a_1(n) = G^{-1}(1/n)$ and note that this yields an increasing, positive sequence which tends to infinity as $n\to\infty$. Furthermore,
\[
\sum_{n=1}^\infty G(a_1(n)) = \sum_{n=1}^\infty \frac{1}{n} = \infty,
\] 
and hence the Borel--Cantelli Lemma implies that
\begin{equation}\label{rapid_lower_bound}
\limsup_{n\to\infty}\frac{|H(n)|}{G^{-1}(1/n)} \geq 1 \mbox{ a.s.}
\end{equation}
Now, because $G^{-1}(1/x) \in \mu$--SSV, there exists $\Delta>0$ such that 
\begin{equation}\label{G_inv_SSV}
\lim_{n\to\infty}\frac{G^{-1}(1/n \mu^\delta(n))}{G^{-1}(1/n)} = 1, \mbox{ for each }\delta \in [0,\Delta].
\end{equation}
Choosing $a_2(n) = G^{-1}(1/n\mu^{\delta^*}(n))$ for each $n > 1$ once more yields a positive, increasing and divergent sequence (where $\delta^*$ is chosen to be the number in $(0,\Delta]$ whose existence was assumed in \eqref{mu_sum_finite}). Thus 
\[
\sum_{n=1}^\infty G(a_2(n))< \infty,
\] 
by hypothesis. It follows from \eqref{G_inv_SSV} that
\[
\limsup_{n\to\infty}\frac{|H(n)|}{G^{-1}(1/n)} = \limsup_{n\to\infty}\frac{|H(n)|}{G^{-1}(1/n \mu^{\delta^*}(n))} \leq 1 \mbox{ a.s.}
\]
Therefore, combining the above inequality with \eqref{rapid_lower_bound}, we have
\[
\limsup_{n\to\infty}\frac{|H(n)|}{G^{-1}(1/n)} = 1 \mbox{ a.s.,}
\]
as claimed.
\section{Proof of Theorem \ref{rv_tails_theorem}}
Suppose that $(i.)$ holds and let $(a(n))_{n \geq 0}$ be an arbitrary positive increasing sequence which tends to infinity. In the notation of \eqref{S_defn}, we have
\[
S(a,K) = \sum_{n=1}^\infty \left\{ 1 - F(K a(n))+ F(a(n)) \right\}, \mbox{ for each $K > 0$}.
\]
Define $G(x) = 1 - F(x)$ for $x \in (-\infty,\infty)$ and 
\[
S_N(a,K) = \sum_{n=1}^N \left\{ G(K a(n))+ F(a(n)) \right\}, \mbox{ for each $K > 0$ and $N \geq 1$}.
\]
Since the summands are non--negative, $\lim_{N\to\infty}S_N(a,K) = S(a,K)$ either converges to a finite limit or to $+\infty$. Now consider the following dichotomy: either 
\begin{equation}\label{case_1}
S(a,K) = \infty \mbox{ for each }K \in (0,\infty),
\end{equation}
or 
\begin{equation}\label{case_2}
\mbox{there exists a }K^* \in (0,\infty) \mbox{ such that } S(a,K^*) < \infty.
\end{equation}
If \eqref{case_1} holds, then 
\[
\limsup_{n \to \infty}\frac{|H(n)|}{a(n)} = \infty \mbox{ a.s.,}
\]
by a simple application of the Borel--Cantelli Lemma. We claim that if \eqref{case_2} holds, then 
\[
S(a,K) < \infty \mbox{ for each }K \in (0,\infty),
\]
and hence that 
\[
\limsup_{n \to \infty}\frac{|H(n)|}{a(n)} = 0 \mbox{ a.s.}
\]
Let $K \in (0,\infty)$ be arbitrary. By hypothesis, $\lim_{n\to\infty}G(Ka(n))/F(-Ka(n)) = \infty$ and there exists $N_1 > 2$ such that $F(-Ka(n)) < G(K a(n))$ for all $ n \geq N_1$. Thus
\begin{align}\label{S_N_first_est}
S_N(a,K) &= \sum_{n=1}^{N_1-1}\left\{ G(Ka(n)) + F(-Ka(n)) \right\} + \sum_{n=N_1}^{N}\left\{ G(Ka(n)) + F(-Ka(n)) \right\}\nonumber\\
&< \sum_{n=1}^{N_1-1}\left\{ G(Ka(n)) + F(-Ka(n)) \right\} + 2\sum_{n=N_1}^{N} G(Ka(n)), \mbox{ for each }N > N_1.
\end{align}
Using the regular variation of $G$, we have
\[
\lim_{n\to\infty}\frac{G(K^* a(n))}{G(K a(n))} = \lim_{n\to\infty}\frac{G(K^* a(n))}{G(a(n))}\frac{G(a(n))}{G(K a(n))} = (K^*)^{-\alpha}K^\alpha =: \kappa > 0.
\]
Hence for each $\epsilon\in(0,\kappa)$ there exists an $N_2(\epsilon) > N_1$ such that for all $n \geq N_2(\epsilon)$,
\[
G(Ka(n)) < \frac{1}{\kappa-\epsilon}G(K^* a(n)).
\]
Plugging the estimate above into \eqref{S_N_first_est} yields
\begin{align*}
S_N(a,K) &< \sum_{n=1}^{N_1-1}\left\{ G(Ka(n)) + F(-Ka(n)) \right\} + 2\sum_{n=N_1}^{N_2-1} G(Ka(n)) + \frac{2}{\kappa-\epsilon}\sum_{n=N_2}^{N} G(K^*a(n))\\
&\leq \sum_{n=1}^{N_1-1}\left\{ G(Ka(n)) + F(-Ka(n)) \right\} + 2\sum_{n=N_1}^{N_2-1} G(Ka(n)) + \frac{2}{\kappa-\epsilon}S_N(a,K^*), \mbox{ for each }N > N_2.
\end{align*}
Finally, let $N \to \infty$ in the estimate above to see that $S(a,K) < \infty$ for each $K \in (0,\infty)$.

The arguments needed to tackle cases $(ii.)$ and $(iii.)$ are exactly analogous to those given above and hence the details are omitted.
\section{Proof of Theorem~\ref{theorem.nonlineargrowth}}
If $y$ is the solution of \eqref{eq.y}, then $y$ is given by 
\[
y(n)=r(n)\xi + \sum_{j=1}^n r(n-j)H(j), \quad n\geq 1
\]
where $r$ is the solution of \eqref{eq.r}. From Theorem~\ref{theorem.fluct} it follows from the fact that $\Lambda_a|H|<+\infty$ that 
\[
\limsup_{n\to\infty} \frac{|y(n)|}{a(n)}=:\Lambda_a|y| < +\infty.
\]
Define $z(n)=x(n)-y(n)$ for $n\geq 0$. Then $z(0)=0$ and for $n\geq 0$ we have from 
\eqref{eq.nonx} and \eqref{eq.y} that
\[
z(n+1)=\sum_{j=0}^n k(n-j)[f(x(j))-y(j)] = \sum_{j=0}^n k(n-j)[\phi(x(j))+z(j)],
\]
where $\phi(x):=f(x)-x$. By \eqref{eq.f2} we have that $\phi(x)/x\to 0$ as $|x|\to\infty$ and that $\phi$ is continuous. Hence we have 
\begin{equation}  \label{eq.dynz}
	z(n+1)=G(n+1)+\sum_{j=0}^n k(n-j)z(j), \quad n\geq 0; \quad z(0)=0
\end{equation}
and 
\begin{equation} \label{def.G}
	G(n+1):=\sum_{j=0}^n k(n-j)\phi(x(j)), \quad n\geq 0.
\end{equation}
Therefore, we have from \eqref{eq.xrep} and \eqref{eq.dynz} that $z$ has the representation 
\begin{equation} \label{eq.zrep}
	z(n)=\sum_{j=1}^n r(n-j)G(j)=\sum_{l=0}^{n-1} r(n-l-1) G(l+1), \quad n\geq 1.
\end{equation}
We next seek to estimate $G$, and thereby deduce a linear summation inequality for $z$. Define $|k|_1:=\sum_{j=0}^\infty |k(j)|$, $|r|_1:=\sum_{j=0}^\infty |r(j)|$,  and choose $\epsilon>0$ so that 
$\epsilon|k|_1 |r_1|<1/2$. For every $\epsilon>0$, by the properties of $\phi$, there exists a $\Phi(\epsilon)>0$ such that $|\phi(x)|\leq \Phi(\epsilon)+\epsilon|x|$ for all $x\in \mathbb{R}$. Hence
\begin{align*}
	|G(n+1)|&\leq \sum_{j=0}^n |k(n-j)||\phi(x(j))|
	\leq  \sum_{j=0}^n |k(n-j)|\left\{\Phi(\epsilon) +\epsilon|x(j)|\right\} \\
	&\leq  |k|_1\Phi(\epsilon) +\epsilon\sum_{j=0}^n |k(n-j)||z(j)| + \epsilon\sum_{j=0}^n |k(n-j)||y(j)|.
\end{align*}
Now by defining $c(n)=\sum_{j=0}^n |r(n-j)||k(j)|$, we have
\begin{align*}
	|z(n+1)| &\leq \sum_{l=0}^{n} |r(n-l)| |G(l+1)|\\
	&\leq \sum_{l=0}^{n} |r(n-l)| \left\{|k|_1\Phi(\epsilon) +\epsilon\sum_{j=0}^l |k(l-j)||z(j)| + \epsilon\sum_{j=0}^l |k(l-j)||y(j)|\right\}\\
	&\leq |r|_1|k|_1\Phi(\epsilon)+\epsilon\sum_{l=0}^{n} \sum_{j=0}^l |r(n-l)|  |k(l-j)||y(j)| + \epsilon\sum_{l=0}^{n} \sum_{j=0}^l |r(n-l)|  |k(l-j)||z(j)|\\
	&= |r|_1|k|_1\Phi(\epsilon)+\epsilon\sum_{j=0}^{n} c(n-j) |y(j)| + \epsilon\sum_{j=0}^{n} c(n-j) |z(j)|. 
\end{align*}
Therefore we have for $n\geq 0$ the estimate
\[
|z(n+1)|\leq |r|_1|k|_1\Phi(\epsilon)+\epsilon\sum_{j=0}^{n} c(n-j) |y(j)| + \epsilon\sum_{j=0}^{n} c(n-j) |z(j)|.
\]
Define 
\[
H_2(n+1)=|r|_1|k|_1\Phi(\epsilon)+\epsilon\sum_{j=0}^{n} c(n-j) |y(j)|, \quad n\geq 0
\]
and 
\begin{equation} \label{def.r2}
	r_2(n+1)=\epsilon \sum_{j=0}^{n}  c(n-j) r(j), \quad n\geq 0;\quad r_2(0)=1.
\end{equation}
Let $z_2$ be the solution of the summation equation
\[
z_2(n+1)= H_2(n+1) + \epsilon\sum_{j=0}^{n} c(n-j) z_2(j), \quad n\geq 0;\quad z_2(0)=1.
\]
Then $|z(n)|\leq z_2(n)$ for all $n\geq 0$. Moreover, $z_2$ has the representation 
\begin{equation} \label{eq.z2rep}
	z_2(n)=r_2(n)+\sum_{j=1}^n r_2(n-j)H_2(j), \quad n\geq 1.
\end{equation}
To determine the asymptotic behaviour of $z_2$, we use the 
representation \eqref{eq.z2rep}. This requires knowledge of the asymptotic behaviour of $H_2$.
We also need to check that $r_2$ is summable. Since $c(n)\geq 0$ for each $n\geq 0$ and 
\[
\epsilon \sum_{n=0}^\infty c(n)=\epsilon|k|_1 |r|_1<\frac{1}{2} 
\]
it follows that $r_2$ is summable, and of course $r(n)\to 0$ as $n\to\infty$. Since 
$c$ is summable, and $\Lambda_a|y|<+\infty$, we have from Lemma~\ref{lemma.convdiva} that 
\[
\limsup_{n\to\infty}
\frac{1}{a(n)}\left|\sum_{j=1}^{n} c(n-j) |y(j)|\right|\leq \sum_{j=0}^\infty c(j) \cdot \Lambda_a|y|.
\]
Therefore as  
\[
H_2(n+1)= |r|_1|k|_1\Phi(\epsilon)+\epsilon c(n) |y(0)|+\epsilon\sum_{j=1}^{n} c(n-j) |y(j)|,
\]  
and $c(n)\to 0$ as $n\to\infty$, we have 
\[
\limsup_{n\to\infty} \frac{|H_2(n+1)|}{a(n)}\leq \epsilon\sum_{j=0}^\infty c(j) \cdot \Lambda_a|y|.
\]
Also, because $a$ is increasing, 
\[
\Lambda_a|H_2|:=
\limsup_{n\to\infty} \frac{|H_2(n)|}{a(n)}\leq \epsilon\sum_{j=0}^\infty c(j)\cdot \Lambda_a|y|.
\]
Since $r_2$ is summable and $\Lambda_a|H_2|<+\infty$, applying Lemma~\ref{lemma.convdiva} once more yields   
\begin{equation}\label{limsup_est_H_2}
\limsup_{n\to\infty} \frac{1}{a(n)}\left|\sum_{j=1}^n r_2(n-j)H_2(j)\right| \leq \sum_{j=0}^\infty r_2(j) \Lambda_a|H_2|
\leq \sum_{j=0}^\infty r_2(j) \cdot \epsilon\sum_{j=0}^\infty c(j) \cdot \Lambda_a|y|.
\end{equation}
We wish to identify a bound on the right hand side in terms of $\epsilon$ and quantities which are explicitly $\epsilon$--independent. We start by estimating the sum of $r_2$. Since $r_2(n)\geq 0$ for all $n\geq 0$, by \eqref{def.r2} and the fact that $r_2(0)=1$ we have 
\[
\sum_{n=0}^\infty r_2(n)-1=
\sum_{n=0}^\infty r_2(n+1)= \sum_{n=0}^\infty \sum_{j=0}^{n} \epsilon c(n-j) r_2(j)
=\epsilon \sum_{n=0}^\infty c(n)\cdot \sum_{n=0}^\infty r_2(n).
\]
Hence as $\sum_{n=0}^\infty c(n)=|k|_1 |r|_1$,
\[
\sum_{n=0}^\infty r_2(n)=\frac{1}{1-\epsilon |k|_1 |r|_1},
\]
and combining this with \eqref{limsup_est_H_2} yields  
\[
\limsup_{n\to\infty} \frac{1}{a(n)}\left|\sum_{j=1}^n r_2(n-j)H_2(j)\right| 
\leq   \frac{1}{1-\epsilon |k|_1 |r|_1} \cdot \epsilon |k|_1 |r|_1 \Lambda_a|y|.
\] 
Thus by \eqref{eq.z2rep} and the fact that $|z(n)|\leq z_2(n)$, we get 
\[
\limsup_{n\to\infty} \frac{|z(n)|}{a(n)}\leq   \frac{1}{1-\epsilon |k|_1 |r|_1} \cdot \epsilon |k|_1 |r|_1 \Lambda_a|y|.
\]
Since $y$, $k$, $r$, $z$, and $a$ are $\epsilon$--independent, letting $\epsilon\to 0^+$ gives 
\[
\limsup_{n\to\infty} \frac{|z(n)|}{a(n)}=0,
\]
which proves the result.
\section{Proof of Theorem~\ref{theorem.nongrowth2}}
Suppose that $H\in BG_{a,\lambda}$. By Theorem \ref{theorem.growth3}, the solution $y$ of \eqref{eq.y} obeys
\[
\frac{y(n)}{a(n)} \sim (\lambda_a y)(n) + \sum_{j=1}^n r(j) \lambda^j (\lambda_a y)(n-j), \quad \text{as }n\to\infty.
\]
Furthermore, by Theorem \ref{theorem.nonlineargrowth}, $x(n)/a(n) \sim y(n)/a(n)$ as $n\to\infty$. Thus we immediately have that \eqref{x_asym_relation_nonlinear} holds and hence that $x \in BG_{a,\lambda}$. 

Conversely, suppose that $x \in BG_{a,\lambda}$. For each $n \geq 0$,
\begin{align*}
\frac{H(n+1)}{a(n+1)} &= \frac{x(n+1)}{a(n+1)} - \sum_{j=0}^n k(n-j)\frac{f(x(j))}{a(n+1)} \\
&= \frac{x(n+1)}{a(n+1)} - \sum_{j=0}^n \frac{k(n-j)}{a(n+1)} \left\{ f(x(j)) - x(j) \right\} + \sum_{j=0}^n k(n-j) \frac{x(j)}{a(n+1)}.
\end{align*}
Hence
\begin{multline}\label{linearisation_est_1}
\left| \frac{H(n+1)}{a(n+1)} - (\lambda_a x)(n+1) + \sum_{j=0}^n k(j)\lambda^{j+1}(\lambda_a x)(n-j-1) \right| \leq \left| \frac{x(n+1)}{a(n+1)} - (\lambda_a x)(n+1) \right|\\ + \left| \sum_{j=0}^n \frac{k(n-j)}{a(n+1)} \left\{ x(j) - f(x(j)) \right\} \right| + \left| \sum_{j=0}^n k(j)\lambda^{j+1}(\lambda_a x)(n-j-1) - \sum_{j=0}^n k(n-j) \frac{x(j)}{a(n+1)} \right|.
\end{multline}
The first term on the right--hand side of \eqref{linearisation_est_1} tends to zero as $n\to\infty$ by hypothesis and the final term tends to zero as $n\to\infty$ by the same argument used in Theorem \ref{theorem.growth3}. Thus it remains to show that
\[
\lim_{n\to\infty}\left| \sum_{j=0}^n \frac{k(n-j)}{a(n+1)} \left\{ x(j) - f(x(j)) \right\} \right| = 0.
\] 
Let $\phi(x) = x - f(x)$ for each $x \in \mathbb{R}$ and note that $\phi$ is continuous by hypothesis. By dint of \eqref{eq.f2}, for each $\epsilon>0$ there exists a $\Phi(\epsilon)>0$ such that
\[
|\phi(x)| \leq \epsilon x + \Phi(\epsilon), \quad \text{for each }x \in \mathbb{R}.
\] 
Now estimate as follows
\begin{align*}
\left| \sum_{j=0}^n \frac{k(n-j)}{a(n+1)} \left\{ x(j) - f(x(j)) \right\} \right| &\leq  \sum_{j=0}^n \left|\frac{k(n-j)}{a(n+1)}\right| \left\{ \epsilon |x(j)|+ \Phi(\epsilon) \right\} \leq \epsilon |k|_1 \bar{x} + \frac{\Phi(\epsilon) |k|_1}{|a(n+1)|},
\end{align*}
where $\bar{x}>0$ uniformly bounds $x(j)/a(n+1)$ (which is possible since $x \in BG_{a,\lambda}$). Now letting $n\to\infty$ and then $\epsilon\downarrow 0$ in the estimate above yields the desired conclusion.
\bibliography{linear_summation_volterra}
\bibliographystyle{abbrv}
\end{document}